\documentclass[a4paper,12pt]{amsart}
\usepackage{amssymb, enumerate,psfrag,graphicx,amsfonts,amsrefs,amsthm,mathrsfs,amsmath,amscd,version,graphicx}
\usepackage{tikz,cleveref,caption}
\usepackage[e]{esvect}
\usetikzlibrary{graphs,graphs.standard}

\usepackage{xcolor}
\usepackage{tikz-cd}
\usepackage{tikz}
\usetikzlibrary{arrows}
\usetikzlibrary{decorations.pathreplacing}
\tikzset{negated/.style={
        decoration={markings,
            mark= at position 0.5 with {
                \node[transform shape] (tempnode) {$\backslash$};
            }
        },
        postaction={decorate}
    }
}
\tikzset{
    vertex/.style={draw,circle,fill=black,thick, inner sep=1.5pt,minimum size=4pt,
    },
      vtx/.style={draw,circle,thick, inner sep=1.5pt,minimum size=4pt,
    },    edge/.style={},
    dedge/.style = {->,> = latex',thick,bend right=45},
    ddedge/.style = {postaction={decorate,decoration={
        markings,
        mark=at position .72 with {\arrow[#1]{Stealth[length=6pt]}, ultra thick}
      }}}
}
\definecolor{asparagus}{rgb}{0.53, 0.66, 0.42}
\usetikzlibrary{decorations.markings}
\usetikzlibrary{decorations.text}
\Crefname{equation}{Eq.}{Eqs.}
\newcommand{\diag}{\operatorname{diag}}

\newtheorem{theo}{Theorem}[section]
\newtheorem{theorem}[theo]{Theorem}
\newtheorem{lemma}[theo]{Lemma}
\newtheorem{proposition}[theo]{Proposition}
\newtheorem{corollary}[theo]{Corollary}
\newtheorem{definition}[theo]{Definition}

\theoremstyle{definition}
\newtheorem{remark}[theo]{Remark}
\newtheorem{example}[theo]{Example}

\numberwithin{equation}{section}

\begin{document}

\title[Degree similar and net-degree similar signed graphs]{Degree similar and net-degree similar signed graphs}

\author[M. Cavaleri]{Matteo Cavaleri}
\address{Matteo Cavaleri, Universit\`{a} Niccol\`{o} Cusano, Dipartimento di Ingegneria, Via Don Carlo Gnocchi, 3 00166 Roma, Italia}
\email{matteo.cavaleri@unicusano.it \textrm{(Corresponding Author)}}

\author[A. Donno]{Alfredo Donno}
\address{Alfredo Donno, Universit\`{a}  Niccol\`{o} Cusano, Dipartimento di Ingegneria, Via Don Carlo Gnocchi, 3 00166 Roma, Italia}
\email{alfredo.donno@unicusano.it}

\begin{abstract}
Two graphs are degree similar if their adjacency matrices and their degree
matrices are simultaneously similar, a notion recently introduced by
Godsil and Sun. We extend it to signed graphs in two ways, degree
similarity and net-degree similarity, according to whether the degree or
the net-degree matrix is used. We prove that signed trees are determined
by their degree similarity class, but not by their net-degree similarity
class. We then develop constructions of (net-)degree similar signed graphs,
based on a signed version of Godsil--McKay switching, on partial
transposition and on block extensions. As a byproduct, we obtain infinite families both of pairs of net-degree
   similar unicyclic signed graphs with non-isomorphic underlying graphs, and
   of unicyclic signed graphs, not necessarily regular, which are not determined by
   their degree similarity class. This settles
the net-degree version of a problem of Godsil and Sun on unicyclic
graphs.
\end{abstract}

\maketitle

\begin{center}
\end{center}

\section{Introduction}
In this paper we extend to the setting of signed graphs the notion of \emph{degree similarity},
introduced and investigated in \cite{GS} for unsigned graphs.\\
\indent A central problem in spectral graph theory is to  understand which graphs are
determined by their spectrum. Schwenk \cite{schwenk} showed that almost all
trees are not, and the switching technique of Godsil and McKay \cite{GM} has
become one of the main tools for constructing cospectral graphs. Nevertheless,
van Dam and Haemers \cite{DH} conjectured that almost all graphs are
determined by their spectrum.\\
\indent In this context, Wang, Li, Lu and Xu \cite{WLLX} considered a stronger
notion than cospectrality, and studied graphs determined by their
generalized characteristic polynomial $\det(tI-A+\mu D)$, introduced
in \cite{CDS}, where $A$ and $D$ are the adjacency and degree matrices, respectively; two
graphs sharing this polynomial are cospectral with respect to the adjacency,
the Laplacian and the signless Laplacian matrix, and have the same degree
sequence. In \cite[Section~5]{WLLX} the authors observe that if the
adjacency and degree matrices of two graphs are conjugate by the same
(orthogonal) matrix, then the two graphs have the same generalized
characteristic polynomial, and they ask whether the converse holds.
As observed in \cite{Wan}, a negative answer to this question is already
contained in \cite{DM}.

In \cite{GS}, pairs of graphs whose adjacency matrices and degree matrices are
simultaneously similar were called \emph{degree similar}, and it was shown
that, while there are infinitely many pairs of non-isomorphic trees with the
same generalized characteristic polynomial, no two non-isomorphic trees are
degree similar (using a result of McKay \cite{McK}). In particular, trees
are determined by their degree similarity class, which shows that degree
similarity is a genuinely stronger notion than generalized cospectrality and
is of independent interest.\\
\indent While \cite{GS} provides several tools for constructing degree similar pairs
of graphs, many questions
remain open.
For instance, \cite[Problem~27]{GS} asks whether non-isomorphic unicyclic
graphs can be degree similar. This question is also addressed in
\cite{FX}, where several families of unicyclic graphs are shown to be
determined by their degree similarity class, and a negative answer has
been given in the recent preprint \cite{ZF9oct26}.
 In \cite{FX} the authors  introduce \emph{orthogonally
   degree-similar} graphs (see also \cite{FWZ}), requiring the similarity to be realized by an
   orthogonal matrix, as in the original question of \cite{WLLX}.
   Since adjacency and degree matrices are symmetric, this notion coincides
with degree similarity, also for signed graphs (see Lemma~\ref{lem:orto}),
as was  implicit in \cite{GS}. Degree similarity is also related to the multivariate graph spectra
introduced by Wang and Zhao \cite{BL0} and further studied in \cite{BL1}.

Signed graphs \cite{zasign} have a well-developed spectral theory (see,
for instance, \cite{BCKW,zasbib}), and many classical problems and tools
have a natural signed counterpart. To mention only those relevant for this paper, a signed version of
Godsil--McKay switching was developed in \cite{cos}, partial transposition
was used in \cite{PT} to construct cospectral signed graphs, and spectral
characterization problems for signed graphs were studied, for example, in
\cite{BP,BS}. We refer to \cite{Zaslavsky,zasbib} for further background on
signed graphs.
In general, spectral determination is harder for
signed graphs than for unsigned ones: a single underlying graph carries
exponentially many signatures, and several spectral invariants of unsigned
graphs lose their direct combinatorial
meaning in the presence of negative edges.
For this reason, stronger notions of cospectrality for signed graphs have
also been considered in the literature; see, for instance, \cite{JWZ,GWL} for
the generalized spectrum of signed graphs, and \cite{Hun} for signed graphs
that are cospectral with respect to both the adjacency and the Laplacian
matrix. However, to the best of our knowledge, degree similarity has not
been studied for signed graphs, and this is the subject of the present
paper. \\
\indent A signed graph $\Gamma$ carries two natural notions of vertex degree: the
degree of a vertex in the underlying graph, and its \emph{net-degree}, that
is, the number of positive edges minus the number of negative edges incident
with it. Accordingly, there are two diagonal matrices associated with
$\Gamma$: the degree matrix $D_\Gamma$ and the net-degree matrix
$D_\Gamma^{\pm}$. The notion of degree similarity can therefore be extended
to signed graphs in two ways: \emph{degree similarity}, namely the
simultaneous similarity of the adjacency matrices and of the degree matrices, and \emph{net-degree similarity}, namely the simultaneous
similarity of the adjacency matrices and of the net-degree matrices. When the signed graphs are all-positive, the degree and the net-degree
matrices coincide, and both notions reduce to degree similarity in the sense
of \cite{GS}. We show that degree similarity is invariant under switching, and implies
   cospectrality with respect to both the adjacency and the Laplacian
   matrix; hence it is natural to ask whether it determines a signed graph
   up to switching isomorphism. On the other hand, net-degree similarity
   depends on the chosen representative of the switching class, and implies
   cospectrality with respect to both the adjacency and the net Laplacian
   matrix.

        Our main contributions are the following. We prove that every signed
graph which is degree similar to a signed tree is itself a signed tree;
combined with the result of \cite{GS} on unsigned trees, this implies that
signed trees are determined by their degree similarity class. While the
first statement is immediate for unsigned graphs, it is not for signed
ones, and indeed it already fails for unicyclic signed graphs: a balanced
cycle of even length is degree similar to the disjoint union of two
signed cycles of half length, one balanced and one unbalanced. In
particular, in contrast with the unsigned case (see \cite{ZF9oct26}), unicyclic signed graphs need not be determined by their degree
similarity class. We then develop a set of tools for constructing
pairs of (net-)degree-similar signed graphs, obtained by adapting to this
setting Godsil--McKay switching, partial transposition, and block
constructions; when applied to the unsigned case, they include and extend several
constructions from \cite{GS}. As a byproduct of these constructions, we obtain infinite families of
unicyclic signed graphs with properties which have no counterpart among
unsigned graphs. For instance, the signed local switching produces, already on  nine
vertices, a pair of net-degree similar unicyclic signed graphs with non-isomorphic
underlying graphs, answering the net-degree
analogue of \cite[Problem~27]{GS}. Moreover, the rooted product turns this pair, as well as the
degree similar pair of cycles mentioned above, into infinite families,
yielding in particular unicyclic signed graphs, also non-regular, which
are not determined by their degree similarity class.

\noindent In  Section~\ref{sec:2} we recall some preliminaries on signed graphs.\\
In  Section~\ref{sec:3} we formally introduce degree similarity and
net-degree similarity, and we give a characterization of both notions in
terms of traces of words in the adjacency and (net-)degree matrices
(Proposition~\ref{prop:specht}); it follows that (net-)degree similar signed graphs have the same generalized Zagreb (net)-index.
These criteria are useful to certify that two signed graphs are not (net-)degree
similar, as we do for small pairs of signed graphs with the same
generalized characteristic polynomial (Example~\ref{exa:mu7}).
 We show that there exist
non-isomorphic trees admitting signatures that are net-degree similar
(Example~\ref{exa:-1}), whereas signed trees are determined by their
degree similarity class (Theorem~\ref{theo:tree}), and that unicyclic
signed graphs need not be (Example~\ref{ex:uni}).\\
 In Section~\ref{sec:4} we introduce additional conditions on the signed
version of Godsil--McKay switching developed in \cite{cos}, which ensure
that not only the adjacency spectrum, but also (net-)degree similarity is
preserved. In the unsigned case, our conditions extend the GM$^*$
property of Haemers and Spence~\cite{HS}, and include the construction
of~\cite[Lemma~9]{GS} as a strict special case (see
Subsection~\ref{sub:4.4}). In the signed setting, the construction produces, already on nine
vertices, net-degree similar unicyclic signed graphs with non-isomorphic
underlying graphs (Example~\ref{exa:uni}).\\
In {Section~\ref{sec:5}}, the partial transposition techniques
introduced in \cite{PT} to construct cospectral signed graphs are refined so
as to preserve the (net-)degree similarity class. In particular,  many of the families of cospectral pairs given in \cite{PT}
are in fact (net-)degree similar.\\
In  Section~\ref{sec:6}, we develop tools for producing new pairs of
(net-)degree-similar signed graphs from a given one, whose adjacency
matrices are used as blocks of the adjacency matrices of the new pairs. We obtain several constructions,
which recover, for instance, some of those given in \cite[Section~8]{GS}.
In particular, (net-)degree similarity is preserved by rooted products,
provided that the rooted signed graphs attached to vertices of the same
(net-)degree are the same; applied to the pairs of Examples~\ref{ex:uni}
and~\ref{exa:uni}, these constructions yield infinitely many unicyclic signed graphs,
also non-regular, which are not determined by their degree similarity
class, as well as infinitely many pairs of net-degree similar unicyclic signed
graphs with non-isomorphic underlying graphs (Example~\ref{exa:rooted}).

\section{Preliminaries}\label{sec:2}
A \emph{signed graph} is a pair  $\Gamma=(G, \sigma)$,
where $G= (V_G, E_G)$ is a graph (which, throughout this paper, will always be assumed simple), with vertex set $V_G$ and edge set $E_G$,
called the \emph{underlying graph}, and $\sigma \colon E_G \to \{\pm 1\}$ is a \emph{signature},
assigning a positive or negative sign to each edge.
Two vertices $v,w\in V_G$ are adjacent, and we write $v\sim w$, if $\{v,w\}\in E_G$. If this is the case, we say that the vertices $v$ and $w$ are endpoints of the edge $\{v,w\}$ and that the vertex $v$ (as well as $w$) and the edge $\{v,w\}$ are incident.
With a slight abuse of notation, we will also write $V_\Gamma$ and
$E_\Gamma$ for $V_G$ and $E_G$, respectively.
For a vertex $v\in V_\Gamma$, the \emph{neighborhood} of $v$ is the set
$N(v)=\{w\in V_G : v\sim w\}$, and the \emph{degree} of $v$ is
$d_\Gamma(v)=|N(v)|$.
We also define
$N^+(v)=\{ w \in N(v) : \sigma(v,w) = 1 \}$ and
$d^+_\Gamma(v) = |N^+(v)|$;
similarly, we define
$N^-(v)=\{ w \in N(v) : \sigma(v,w) = -1 \}$ and
$d^-_\Gamma(v) = |N^-(v)|$. Finally,
 we set $d^\pm_\Gamma(v) = d^+_\Gamma(v) - d^-_\Gamma(v)$,
which is called the \emph{net-degree} of the vertex $v$ in $\Gamma$.
When it is clear from the context, we will omit the subscript $\Gamma$.
When we write $d_\Gamma^{(\pm)}(v)$, in words \emph{(net-)degree} of $v$, we simultaneously refer to the degree
$d_\Gamma(v)$ and the net-degree $d_\Gamma^\pm(v)$. The signed graph $\Gamma$ is \emph{$k$-(net-)regular} if
$d^{(\pm)}_\Gamma(v)=k$ for every $v\in V_\Gamma$, and \emph{(net-)regular}
if it is $k$-(net-)regular for some $k$. When the signature $\sigma$ is all-positive, we will often identify the signed
graph $(G,\sigma)$ with its underlying graph $G$.
Two signed graphs $\Gamma_1=(G_1,\sigma_1)$ and $\Gamma_2=(G_2,\sigma_2)$
are isomorphic if there exists an isomorphism
$\phi\colon G_1\to G_2$ of their underlying graphs, such that $\sigma_2(\phi(u),\phi(v))=\sigma_1(u,v)$
for every edge $\{u,v\}$ of $G_1$; such a $\phi$ is called an isomorphism
from $\Gamma_1$ to $\Gamma_2$.
Two signatures $\sigma$ and $\sigma'$ on the same underlying graph $G$ are
said to be \emph{switching equivalent} if there is a subset $W\subset V_G$ such that $\sigma'$ is obtained from
$\sigma$ by reversing the sign of every edge with exactly one endpoint in
$W$. In this case, we will also say that the signed graphs $(G,\sigma)$ and
$(G,\sigma')$ are switching equivalent.
Two signed graphs $\Gamma_1$ and $\Gamma_2$ are
\emph{switching isomorphic} if there exists a signed graph $\Gamma_1'$ that is
 switching
equivalent to $\Gamma_1$  and isomorphic to
$\Gamma_2$.
A signature $\sigma$, and by extension the signed graph $\Gamma=(G, \sigma)$,
is called \emph{balanced} (see \cite{Harary}) if   every closed walk  in $\Gamma$ contains an even number of negative edges, or equivalently, if
 $\sigma$ is switching equivalent to the all-positive signature.

 We denote by $I_n$ and $J_n$ the $n\times n$ identity matrix and the
$n\times n$ all-ones matrix, respectively. When the context is clear, we will omit the subscript and simply write $I$
and $J$. A matrix $M$ is said to be \emph{block diagonal} if it has the form
$M=\diag(M_1,M_2,\ldots,M_k)$, where $M_1,\ldots,M_k$ are square matrices.
In other words, $M$ has $M_1,\ldots,M_k$ as its diagonal blocks and zero
blocks elsewhere.
A real matrix $M$ is said to be \emph{normal} if $MM^t=M^tM$, where $M^t$ denotes the transpose of $M$.
\\Once an ordering of the vertices $v_1,v_2,\ldots,v_n$ of a signed graph
$\Gamma=(G,\sigma)$ is fixed, its \emph{adjacency matrix} $A_{\Gamma}$ is determined.
It is the $n\times n$ symmetric matrix whose $(i,j)$-entry is
$\sigma(v_i,v_j)$ if $v_i\sim v_j$, and $0$ otherwise.  Similarly, the \emph{degree matrix}
$D_{\Gamma}$ and the \emph{net-degree matrix} $D^{\pm}_{\Gamma}$ are determined;
they are diagonal matrices whose diagonal entries are given by
$(D_{\Gamma})_{ii}=d_{\Gamma}(v_i)$ and
$(D^{\pm}_{\Gamma})_{ii}=d^{\pm}_{\Gamma}(v_i)$, respectively.
When we write $D^{(\pm)}_\Gamma$, we simultaneously refer to the degree matrix  $D_{\Gamma}$ and the net-degree matrix  $D^{\pm}_{\Gamma}$.
A change of vertex ordering or an isomorphism between signed graphs induces
the same permutation similarity on both $A_\Gamma$ and $D^{(\pm)}_\Gamma$;
that is, there exists a permutation matrix $P$ such that
$P^{-1}A_\Gamma P$ and $P^{-1}D^{(\pm)}_\Gamma P$ are the corresponding
matrices for the new ordering or the isomorphic signed graph.
On the other hand, if two signatures $\sigma$ and $\sigma'$ on the same
underlying graph $G$ are switching equivalent, with
$\Gamma=(G,\sigma)$ and $\Gamma'=(G,\sigma')$, then there exists a
\emph{switching matrix} $S$, that is, a diagonal matrix with diagonal entries
in $\{\pm1\}$, such that
$A_{\Gamma'}=S^{-1} A_\Gamma S$; notice that  $D_{\Gamma'}=S^{-1} D_\Gamma S=D_\Gamma$ but, in general, $D^{\pm}_{\Gamma'}\neq D^{\pm}_\Gamma $.
Therefore, if two signed graphs $\Gamma_1$ and $\Gamma_2$ are switching isomorphic, then, for
any choice of vertex orderings, there exist a permutation matrix $P$ and a
switching matrix $S$ such that
$A_{\Gamma_2}=(PS)^{-1}A_{\Gamma_1}(PS)$ and
$D_{\Gamma_2}=(PS)^{-1}D_{\Gamma_1}(PS)
=P^{-1}D_{\Gamma_1}P$. In particular, the spectrum of the adjacency matrix depends only on the switching
isomorphism class of the signed graph.
\\Finally, we recall the definitions of the \emph{Laplacian matrix}
$L_{\Gamma}$ and the \emph{net-Laplacian matrix} $L^\pm_{\Gamma}$ of a signed graph
$\Gamma$, namely
$L_{\Gamma}=D_{\Gamma}-A_{\Gamma}$ and
$L^\pm_{\Gamma}=D^{\pm}_{\Gamma}-A_{\Gamma}$.
If two signed graphs are switching
isomorphic, then also their Laplacian
matrices are conjugate by a matrix of the form
$PS$.

\section{(Net-)degree similarity and determination}\label{sec:3}
We introduce the two main definitions, namely the two natural possible
generalizations to signed graphs of the concept of degree similarity.
\begin{definition}\label{def:prima}
Let $\Gamma_i$  be a signed graph, and let
$A_{\Gamma_i}$ and $D_{\Gamma_i}$ denote, respectively, its adjacency
and degree matrix with respect to a common ordering of the vertices, for
$i=1,2$. We say that $\Gamma_1$ and $\Gamma_2$ are
\emph{degree similar} if there exists a real, invertible matrix $B$ such that
\begin{equation}\label{eq:nnet}
B^{-1}A_{\Gamma_1} B=A_{\Gamma_2} \quad \mbox{ and } \quad B^{-1}D_{\Gamma_1} B=D_{\Gamma_2}.
\end{equation}
\end{definition}
\begin{definition}
Let $\Gamma_i$  be a signed graph, and let
$A_{\Gamma_i}$ and $D^\pm_{\Gamma_i}$ denote, respectively, its adjacency
and net-degree matrix with respect to a common ordering of the vertices, for
$i=1,2$. We say that $\Gamma_1$ and $\Gamma_2$ are
\emph{net-degree similar} if there exists a real, invertible matrix $B$ such that
\begin{equation}\label{eq:net}
B^{-1}A_{\Gamma_1} B=A_{\Gamma_2} \quad \mbox{ and } \quad B^{-1}D^\pm_{\Gamma_1} B=D^\pm_{\Gamma_2}.
\end{equation}
Finally, we say that $\Gamma_1$ and $\Gamma_2$ are \emph{simultaneously degree similar and net-degree similar} if there exists a real, invertible matrix $B$ such that both Eq.~\eqref{eq:nnet} and Eq.~\eqref{eq:net} hold.
\end{definition}

Observe that both definitions are independent of the ordering of the
vertices, which only changes the matrix $B$ by a permutation matrix, and
that two (net-)degree similar signed graphs have the same multiset of
(net-)degrees, since similar diagonal matrices have the same diagonal
entries up to a permutation.
\begin{remark}\label{prop:regular}
Let $\Gamma_1$ be a $k$-(net-)regular signed graph, and let $\Gamma_2$
be a signed graph cospectral with $\Gamma_1$. Then $\Gamma_1$ and
$\Gamma_2$ are (net-)degree similar if and only if $\Gamma_2$ is also
$k$-(net-)regular, since $D^{(\pm)}_{\Gamma_1}=kI$ commutes with every
matrix. In particular, every pair of cospectral $k$-(net-)regular signed
graphs is (net-)degree similar.
\end{remark}

A general characterization of (net-)degree similarity, a direct consequence of a classical theorem of Specht and Wiegmann, is the following; it will be mostly useful in the negative direction, that is, to show that two
signed graphs are not (net-)degree similar.
\begin{proposition}\label{prop:specht}
Two signed graphs $\Gamma_1$ and $\Gamma_2$ are (net-)degree similar if
and only if
$\operatorname{tr}\big(w(A_{\Gamma_1},D^{(\pm)}_{\Gamma_1})\big)
=\operatorname{tr}\big(w(A_{\Gamma_2},D^{(\pm)}_{\Gamma_2})\big)$ for
every word $w$ in two letters.
\end{proposition}
\begin{proof}
If $B$ satisfies Eq.~\eqref{eq:nnet} (resp.\ Eq.~\eqref{eq:net}), then it
conjugates $w(A_{\Gamma_1},D^{(\pm)}_{\Gamma_1})$ into
$w(A_{\Gamma_2},D^{(\pm)}_{\Gamma_2})$ for every word $w$, so that the
two matrices have the same trace.
Conversely, by a theorem of Specht \cite{Specht}, extended to sets of
matrices by Wiegmann \cite{Wiegmann} (see also \cite[Theorem~2.2]{Jing}),
there is a unitary matrix $U$ such that
$A_{\Gamma_2}=U^*A_{\Gamma_1}U$ and
$D^{(\pm)}_{\Gamma_2}=U^*D^{(\pm)}_{\Gamma_1}U$. By
\cite[Theorem~5]{V}, there exists a real orthogonal matrix $Q$ such that
$A_{\Gamma_2}=Q^{-1}A_{\Gamma_1}Q$ and
$D^{(\pm)}_{\Gamma_2}=Q^{-1}D^{(\pm)}_{\Gamma_1}Q$.
\end{proof}
For instance, taking $w=D^{a} A D^{b}A$
one gets that two (net-)degree similar signed graphs have the same
generalized Zagreb (net-)index
$$\sum_{\{u,v\}\in E_\Gamma}\big(d^{(\pm)}_\Gamma(u)^a d^{(\pm)}_\Gamma(v)^b
+d^{(\pm)}_\Gamma(u)^b d^{(\pm)}_\Gamma(v)^a\big),$$ in the sense of
\cite{AI,NK}; in particular this implies that they share the second Zagreb \mbox{(net-)index}
$\sum_{\{u,v\}\in E_\Gamma}d^{(\pm)}_\Gamma(u)d^{(\pm)}_\Gamma(v)$.

If $\Gamma_1$ and $\Gamma_2$ are (net-)degree similar, then clearly
$\det(tI-A_{\Gamma_1}+\mu D^{(\pm)}_{\Gamma_1})=\det(tI-A_{\Gamma_2}+\mu D^{(\pm)}_{\Gamma_2})$,
that is, they are cospectral with respect to every matrix
$A_\Gamma-\mu D^{(\pm)}_\Gamma$. For unsigned graphs, whether the converse
holds is precisely the question raised in \cite{WLLX}; the answer is
negative, see the pair of graphs on nine vertices in \cite{DM} (as
observed in \cite{Wan}) and the infinite family of trees in
\cite[Corollary~4]{GS}. For signed graphs the converse fails for both
notions, with counterexamples on seven and six vertices, see
Example~\ref{exa:mu7}. A crucial point is that, by
Newton's identities, the equality of the polynomials
$\det(tI-A_\Gamma+\mu D^{(\pm)}_\Gamma)$ amounts to the equality of the
traces $\operatorname{tr}\big((A_\Gamma-\mu D^{(\pm)}_\Gamma)^k\big)$ for
all $k$ and $\mu$, i.e.\ of certain sums of traces of words in $A_\Gamma$
and $D^{(\pm)}_\Gamma$, whereas by Proposition~\ref{prop:specht}
(net-)degree similarity amounts to the equality of the trace of every
single word.
\begin{figure}
\begin{centering}
\begin{tikzpicture}[scale=0.7,
  vertex/.style={circle, fill=black, minimum size=5pt, inner sep=0pt},
  pos edge/.style={thin},
  neg edge/.style={thin, dashed},
  lbl/.style={draw=none, fill=none, font=\small}]
\node[vertex] (1) at (0,1.4)    {};
\node[vertex] (2) at (0,-1.4)   {};
\node[vertex] (3) at (-1.6,0)   {};
\node[vertex] (4) at (1.6,0)    {};
\node[vertex] (5) at (-2.4,1.8) {};
\node[vertex] (6) at (2.4,-1.8) {};
\node[vertex] (7) at (3.4,0)    {};
\node[lbl,above] at (0,1.4)     {$$};
\node[lbl,below] at (0,-1.4)    {$$};
\node[lbl,below left] at (-1.6,0)  {$$};
\node[lbl,above right] at (1.6,0)  {$$};
\node[lbl,above] at (-2.4,1.8)  {$$};
\node[lbl,below] at (2.4,-1.8)  {$$};
\node[lbl,right] at (3.4,0)     {$$};
\draw[pos edge] (1) -- (3);
\draw[pos edge] (1) -- (4);
\draw[neg edge] (2) -- (3);
\draw[pos edge] (2) -- (4);
\draw[pos edge] (3) -- (4);
\draw[pos edge] (1) -- (5);
\draw[pos edge] (3) -- (5);
\draw[pos edge] (2) -- (6);
\draw[pos edge] (4) -- (6);
\draw[pos edge] (4) -- (7);
\node[lbl] at (-2.4,-1.8) {$\Gamma_1$};
\end{tikzpicture}
\qquad\qquad
\begin{tikzpicture}[scale=0.7,
  vertex/.style={circle, fill=black, minimum size=5pt, inner sep=0pt},
  pos edge/.style={thin},
  neg edge/.style={thin, dashed},
  lbl/.style={draw=none, fill=none, font=\small}]
\node[vertex] (1) at (0,1.4)    {};
\node[vertex] (2) at (0,-1.4)   {};
\node[vertex] (3) at (-1.6,0)   {};
\node[vertex] (4) at (1.6,0)    {};
\node[vertex] (5) at (-2.4,1.8) {};
\node[vertex] (6) at (2.4,-1.8) {};
\node[vertex] (7) at (3.4,0)    {};
\node[lbl,above] at (0,1.4)     {$$};
\node[lbl,below] at (0,-1.4)    {$$};
\node[lbl,below left] at (-1.6,0)  {$$};
\node[lbl,above right] at (1.6,0)  {$$};
\node[lbl,above] at (-2.4,1.8)  {$$};
\node[lbl,below] at (2.4,-1.8)  {$$};
\node[lbl,right] at (3.4,0)     {$$};
\draw[pos edge] (1) -- (3);
\draw[neg edge] (1) -- (4);
\draw[pos edge] (2) -- (3);
\draw[pos edge] (2) -- (4);
\draw[pos edge] (3) -- (4);
\draw[pos edge] (1) -- (5);
\draw[pos edge] (3) -- (5);
\draw[pos edge] (2) -- (6);
\draw[pos edge] (4) -- (6);
\draw[pos edge] (4) -- (7);
\node[lbl] at (-2.4,-1.8) {$\Gamma_2$};
\end{tikzpicture}
\end{centering}
\caption{The two signed graphs of Example~\ref{exa:mu7} with the same polynomial
$\det(tI-A_\Gamma+\mu D_\Gamma)$ which are not degree similar.}\label{fig:mu7}
\end{figure}
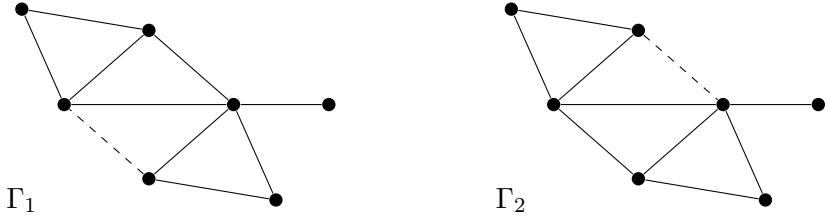

\begin{example}\label{exa:mu7}
Let $\Gamma_1$, $\Gamma_2$ be the signed graphs depicted in
Fig.~\ref{fig:mu7}. Here, and in all subsequent figures, solid lines denote positive edges and dashed lines negative edges.
 A direct computation shows that
$\det(tI-A_{\Gamma_1}+\mu D_{\Gamma_1})=\det(tI-A_{\Gamma_2}+\mu D_{\Gamma_2})$,
whereas
$
\operatorname{tr}\big((A_{\Gamma_1}^2D_{\Gamma_1})^2\big)=1414$ and
$\operatorname{tr}\big((A_{\Gamma_2}^2D_{\Gamma_2})^2\big)=1422;
$
hence $\Gamma_1$ and $\Gamma_2$ are not degree similar by
Proposition~\ref{prop:specht}.
Similarly, the signed graphs $\Gamma_1$ and $\Gamma_2$ of Fig.~\ref{fig:netmu6}
have the same polynomial $\det(tI-A_\Gamma+\mu D^\pm_\Gamma)$, but
$\operatorname{tr}\big((A_{\Gamma_1}^2D^\pm_{\Gamma_1})^2\big)=26$ and
$\operatorname{tr}\big((A_{\Gamma_2}^2D^\pm_{\Gamma_2})^2\big)=18$, so
they are not net-degree similar by Proposition~\ref{prop:specht}.
\end{example}
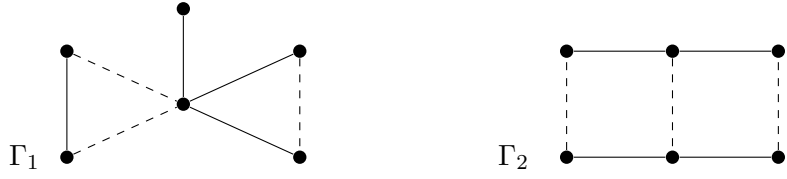
\begin{figure}
\begin{centering}
\begin{tikzpicture}[scale=0.7,
  vertex/.style={circle, fill=black, minimum size=5pt, inner sep=0pt},
  pos edge/.style={thin},
  neg edge/.style={thin, dashed},
  lbl/.style={draw=none, fill=none, font=\small}]
\node[vertex] (6) at (0,0)       {};
\node[vertex] (1) at (-2.2,1)    {};
\node[vertex] (4) at (-2.2,-1)   {};
\node[vertex] (2) at (2.2,1)     {};
\node[vertex] (5) at (2.2,-1)    {};
\node[vertex] (3) at (0,1.8)     {};
\draw[pos edge] (1) -- (4);
\draw[neg edge] (1) -- (6);
\draw[neg edge] (4) -- (6);
\draw[neg edge] (2) -- (5);
\draw[pos edge] (2) -- (6);
\draw[pos edge] (5) -- (6);
\draw[pos edge] (3) -- (6);
\node[lbl] at (-3,-1) {$\Gamma_1$};
\node[lbl] at (3,-1) {$$};
\end{tikzpicture}
\qquad\qquad
\begin{tikzpicture}[scale=0.7,
  vertex/.style={circle, fill=black, minimum size=5pt, inner sep=0pt},
  pos edge/.style={thin},
  neg edge/.style={thin, dashed},
  lbl/.style={draw=none, fill=none, font=\small}]
\node[vertex] (4) at (-2,1)      {};
\node[vertex] (1) at (0,1)       {};
\node[vertex] (5) at (2,1)       {};
\node[vertex] (2) at (-2,-1)     {};
\node[vertex] (6) at (0,-1)      {};
\node[vertex] (3) at (2,-1)      {};
\draw[pos edge] (1) -- (4);
\draw[pos edge] (1) -- (5);
\draw[neg edge] (1) -- (6);
\draw[neg edge] (2) -- (4);
\draw[pos edge] (2) -- (6);
\draw[neg edge] (3) -- (5);
\draw[pos edge] (3) -- (6);
\node[lbl] at (-3,-1) {$\Gamma_2$};
\end{tikzpicture}
\end{centering}
\caption{The two signed graphs of Example~\ref{exa:mu7} with the same polynomial
$\det(tI-A_\Gamma+\mu D_\Gamma^\pm)$ which are not net-degree similar.}\label{fig:netmu6}
\end{figure}

Another natural question is whether (net-)degree similarity is
invariant under switching equivalence. The following result provides the
answer.
\begin{proposition}\label{prop:sw}
If two signed graphs $\Gamma_1$ and
$\Gamma_2$ are degree similar and $\Gamma_1'$ is switching isomorphic to
$\Gamma_1$, then $\Gamma_1'$ and $\Gamma_2$ are degree similar. By contrast, net-degree similarity depends on the representative of the
switching equivalence class.
\end{proposition}
\begin{proof}
If $\Gamma_1'$ is switching isomorphic to $\Gamma_1$, there is a signed
permutation matrix $PS$ with $(PS)^{-1}A_{\Gamma_1'}PS=A_{\Gamma_1}$, and
then $(PS)^{-1}D_{\Gamma_1'}PS=D_{\Gamma_1}$, since switching does not
change the degrees. Hence, if $B$ satisfies Eq.~\eqref{eq:nnet}, so does
$PSB$ in place of $B$ for $\Gamma_1'$ and $\Gamma_2$. By contrast, let
$\Gamma$ be an all-positive $k$-regular graph, $k\ge1$, and let
$\Gamma'$ be obtained from $\Gamma$ by switching at a single vertex,
so that $\Gamma'$ is not net-regular. Then $\Gamma$ is net-degree similar to itself, whereas $\Gamma'$ and
$\Gamma$ are not net-degree similar, by Remark~\ref{prop:regular}.
\end{proof}
Accordingly, we say that a signed graph $\Gamma$ is \emph{determined by its
degree similarity class} if every signed graph that is degree similar to
$\Gamma$ is switching isomorphic to $\Gamma$. We say that $\Gamma$
is \emph{determined by its net-degree similarity class} if every signed graph
that is net-degree similar to $\Gamma$ is isomorphic to $\Gamma$.

Theorem 3 of \cite{GS}  easily implies that every tree is determined by
its degree similarity class in the class of unsigned graphs, since a tree can
be degree similar only to another tree. The situation becomes more involved
for signed graphs.

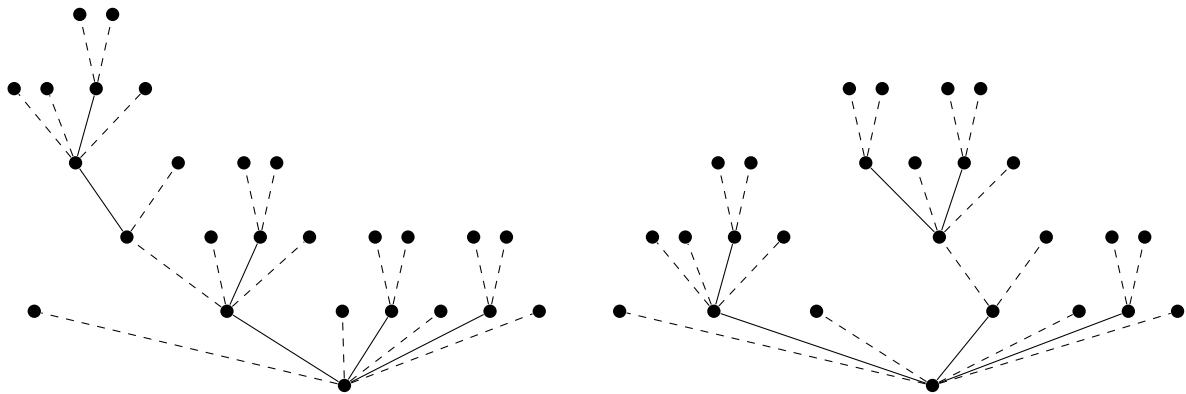
\begin{figure}[h]
\begin{centering}
\begin{tikzpicture}[scale=0.7,
  vertex/.style={circle, fill=black, minimum size=5pt, inner sep=0pt},
  pos edge/.style={thin},
  neg edge/.style={thin, dashed},
  lbl/.style={draw=none, fill=none, font=\small}
]
\node[vertex] (a0) at (6.86, 0.00) {};
\node[vertex] (a1) at (1.00, 1.40) {};
\node[vertex] (a2) at (4.64, 1.40) {};
\node[vertex] (a3) at (6.82, 1.40) {};
\node[vertex] (a4) at (2.75, 2.80) {};
\node[vertex] (a5) at (4.34, 2.80) {};
\node[vertex] (a6) at (7.75, 1.40) {};
\node[vertex] (a7) at (8.68, 1.40) {};
\node[vertex] (a8) at (7.44, 2.80) {};
\node[vertex] (a9) at (8.06, 2.80) {};
\node[vertex] (a10) at (9.61, 1.40) {};
\node[vertex] (a11) at (10.54, 1.40) {};
\node[vertex] (a12) at (9.30, 2.80) {};
\node[vertex] (a13) at (9.92, 2.80) {};
\node[vertex] (a14) at (5.27, 2.80) {};
\node[vertex] (a15) at (6.20, 2.80) {};
\node[vertex] (a16) at (4.96, 4.20) {};
\node[vertex] (a17) at (5.58, 4.20) {};
\node[vertex] (a18) at (1.78, 4.20) {};
\node[vertex] (a19) at (3.72, 4.20) {};
\node[vertex] (a20) at (0.62, 5.60) {};
\node[vertex] (a21) at (1.24, 5.60) {};
\node[vertex] (a22) at (2.17, 5.60) {};
\node[vertex] (a23) at (3.10, 5.60) {};
\node[vertex] (a24) at (1.86, 7.00) {};
\node[vertex] (a25) at (2.48, 7.00) {};

\draw[neg edge] (a0) -- (a1);
\draw[pos edge] (a0) -- (a2);
\draw[neg edge] (a0) -- (a3);
\draw[pos edge] (a0) -- (a6);
\draw[neg edge] (a0) -- (a7);
\draw[pos edge] (a0) -- (a10);
\draw[neg edge] (a0) -- (a11);
\draw[neg edge] (a2) -- (a4);
\draw[neg edge] (a2) -- (a5);
\draw[pos edge] (a2) -- (a14);
\draw[neg edge] (a2) -- (a15);
\draw[pos edge] (a4) -- (a18);
\draw[neg edge] (a4) -- (a19);
\draw[neg edge] (a6) -- (a8);
\draw[neg edge] (a6) -- (a9);
\draw[neg edge] (a10) -- (a12);
\draw[neg edge] (a10) -- (a13);
\draw[neg edge] (a14) -- (a16);
\draw[neg edge] (a14) -- (a17);
\draw[neg edge] (a18) -- (a20);
\draw[neg edge] (a18) -- (a21);
\draw[pos edge] (a18) -- (a22);
\draw[neg edge] (a18) -- (a23);
\draw[neg edge] (a22) -- (a24);
\draw[neg edge] (a22) -- (a25);
\end{tikzpicture}
\begin{tikzpicture}[scale=0.7,
  vertex/.style={circle, fill=black, minimum size=5pt, inner sep=0pt},
  pos edge/.style={thin},
  neg edge/.style={thin, dashed},
  lbl/.style={draw=none, fill=none, font=\small}
]
\node[vertex] (b0) at (5.91, 0.00) {};
\node[vertex] (b1) at (0.00, 1.40) {};
\node[vertex] (b2) at (1.78, 1.40) {};
\node[vertex] (b3) at (3.72, 1.40) {};
\node[vertex] (b4) at (0.62, 2.80) {};
\node[vertex] (b5) at (1.24, 2.80) {};
\node[vertex] (b6) at (7.05, 1.40) {};
\node[vertex] (b7) at (8.68, 1.40) {};
\node[vertex] (b8) at (6.04, 2.80) {};
\node[vertex] (b9) at (8.06, 2.80) {};
\node[vertex] (b10) at (9.61, 1.40) {};
\node[vertex] (b11) at (10.54, 1.40) {};
\node[vertex] (b12) at (9.30, 2.80) {};
\node[vertex] (b13) at (9.92, 2.80) {};
\node[vertex] (b14) at (2.17, 2.80) {};
\node[vertex] (b15) at (3.10, 2.80) {};
\node[vertex] (b16) at (1.86, 4.20) {};
\node[vertex] (b17) at (2.48, 4.20) {};
\node[vertex] (b18) at (4.65, 4.20) {};
\node[vertex] (b19) at (5.58, 4.20) {};
\node[vertex] (b20) at (4.34, 5.60) {};
\node[vertex] (b21) at (4.96, 5.60) {};
\node[vertex] (b22) at (6.51, 4.20) {};
\node[vertex] (b23) at (7.44, 4.20) {};
\node[vertex] (b24) at (6.20, 5.60) {};
\node[vertex] (b25) at (6.82, 5.60) {};
\node[lbl,above] at (-1, 5.60) {$$};
\draw[neg edge] (b0) -- (b1);
\draw[pos edge] (b0) -- (b2);
\draw[neg edge] (b0) -- (b3);
\draw[pos edge] (b0) -- (b6);
\draw[neg edge] (b0) -- (b7);
\draw[pos edge] (b0) -- (b10);
\draw[neg edge] (b0) -- (b11);
\draw[neg edge] (b2) -- (b4);
\draw[neg edge] (b2) -- (b5);
\draw[pos edge] (b2) -- (b14);
\draw[neg edge] (b2) -- (b15);
\draw[neg edge] (b6) -- (b8);
\draw[neg edge] (b6) -- (b9);
\draw[pos edge] (b8) -- (b18);
\draw[neg edge] (b8) -- (b19);
\draw[pos edge] (b8) -- (b22);
\draw[neg edge] (b8) -- (b23);
\draw[neg edge] (b10) -- (b12);
\draw[neg edge] (b10) -- (b13);
\draw[neg edge] (b14) -- (b16);
\draw[neg edge] (b14) -- (b17);
\draw[neg edge] (b18) -- (b20);
\draw[neg edge] (b18) -- (b21);
\draw[neg edge] (b22) -- (b24);
\draw[neg edge] (b22) -- (b25);
\end{tikzpicture}
\end{centering}
\caption{The net-degree similar, nonisomorphic signed trees of Example \ref{exa:-1}.}\label{fig:-1}
\end{figure}

\begin{example}\label{exa:-1}
Consider the signed trees $T_1$ and $T_2$ on $26$ vertices depicted,
respectively, on the left and on the right in Fig.~\ref{fig:-1}.
Their underlying graphs are cospectral trees (and hence so are
$T_1$ and $T_2$), and their signatures are $-1$-net-regular.
Therefore, by Remark~\ref{prop:regular},
they are net-degree similar. Notice that the underlying graphs of $T_1$ and $T_2$ are not isomorphic.
Therefore, these two signed trees are not determined by their net-degree-
similarity classes.
\\   The underlying graphs were found by a computer search among the trees
   admitting a $\pm1$-eigenvector with eigenvalue $1$, which are
   characterized in \cite[Theorem~23]{SS}; these are exactly the trees
   admitting a net-regular signature (see also \cite{SEDIE}).
\end{example}
\begin{lemma}\label{lem:tree}
Two signed trees are degree similar if and only if their underlying graphs are
isomorphic. In contrast, neither implication holds for the analogous
statement involving net-degree similarity.
\end{lemma}
\begin{proof}
Let $T_1$ and $T_2$ be two signed trees. Since every signed tree is balanced,
it is switching equivalent to the all-positive signed tree on the same
underlying graph. Thus, by Proposition~\ref{prop:sw}, we may assume that
$T_1$ and $T_2$ are both all-positive. The result then follows directly from
\cite[Theorem 3]{GS}.
\\On the other hand, the trees $T_1$ and $T_2$ of Example~\ref{exa:-1} are
net-degree similar, while their underlying graphs are not isomorphic. Moreover,
it suffices to consider any signed tree $T_1'$ on the same underlying graph as
$T_1$ that is not net-regular, in order to obtain two signed trees with
isomorphic underlying graphs that are not net-degree similar according to Remark \ref{prop:regular}.
\end{proof}
The claim of Lemma~\ref{lem:tree} alone does not yet guarantee that a
signed tree is determined by its degree similarity class, unlike
\cite[Theorem 3]{GS} for unsigned trees. The only remaining obstruction is the possibility that a signed graph that is not a tree may be degree similar to a tree. We shall prove that this cannot occur.

Before doing so, we record  in Lemma \ref{lem:orto} a standard consequence  of the polar
decomposition, which we prove for completeness.
\begin{lemma}\label{lem:orto}
Let $\Gamma_1$ and $\Gamma_2$ be two (net-)degree similar signed graphs,
with adjacency and (net-)degree matrices taken with respect to a common
ordering of the vertices. Then there exists an orthogonal matrix $Q$ such
that
\begin{equation*}
Q^{t}A_{\Gamma_1} Q=A_{\Gamma_2} \quad \mbox{ and } \quad
Q^{t}D^{(\pm)}_{\Gamma_1} Q=D^{(\pm)}_{\Gamma_2}.
\end{equation*}
\end{lemma}
\begin{proof}
By definition, there exists an invertible real matrix $B$ satisfying
Eq.~\eqref{eq:nnet} (resp.\ Eq.~\eqref{eq:net}). Since $B^{t}B$ is
positive definite, its positive definite square root $(B^{t}B)^{1/2}$ is
well defined and invertible, and we put $Q:=B(B^{t}B)^{-1/2}$. The
matrix $Q$ is orthogonal, since $(B^tB)^{-1/2}$ is symmetric and
commutes with $B^tB$. By Eq.~\eqref{eq:nnet} (resp.\ Eq.~\eqref{eq:net})
and the symmetry of $A_{\Gamma_2}$, the matrix $B^tB$ commutes with
$A_{\Gamma_2}$, hence so does $(B^tB)^{1/2}$, and
$Q^tA_{\Gamma_1}Q=(B^tB)^{1/2}A_{\Gamma_2}(B^tB)^{-1/2}=A_{\Gamma_2}$;
the same holds for $D^{(\pm)}_{\Gamma_2}$.
\end{proof}
In particular, degree similarity coincides with the notion of orthogonal
degree similarity introduced in \cite[Definition~27]{FX} and considered also in \cite{FWZ}.

\begin{theorem}\label{theo:tree}
Let $\Gamma$ and $T$ be degree similar signed graphs. If $T$ is a signed
tree, then $\Gamma$ is a signed tree as well.
\end{theorem}
\begin{proof}
Suppose that $n$ is the number of vertices of $T$. The Laplacian spectra of the two signed
graphs $\Gamma$ and $T$ coincide. This implies that the number of vertices $n$, the number of
edges $n-1$, and  the number of balanced connected components (i.e., nullity of the Laplacian) are the same for
the two signed graphs (see \cite{BS,zasign,Zaslavsky}).
It easily follows  that $\Gamma$ is the union of a signed tree $S$ (and hence
balanced), say on $k$ vertices, with $k\leq n$, and possibly some unbalanced unicyclic graphs.
 By Proposition~\ref{prop:sw}, we may assume, without loss of generality, that
the signed subgraph $S$  is all-positive. We will prove that $k=n$.

By Lemma~\ref{lem:orto}, there exists an orthogonal matrix $Q$ such that
$Q^t A_\Gamma Q= A_T$ and $Q^t D_\Gamma Q= D_T$ and hence also
\begin{equation}\label{eq:lap}
L_{\Gamma}Q=QL_{T}.
\end{equation}
Let $\mathbf{v}$ be the vector of length $n$ whose entries corresponding to the vertices of $S$ are
equal to $1$, while all the remaining entries are zero. We know that $\ker L_{\Gamma}$ has dimension $1$ and is in fact generated by
$\mathbf{v}$. The all-ones vector $\mathbf{1}$ is clearly in the kernel of $L_T$, and
hence also of $Q L_T$. By Eq.~\eqref{eq:lap}, it must also belong to the
kernel of $L_{\Gamma}Q$. This is possible only if $Q\mathbf{1}=c \mathbf{v}$ for some
scalar $c$.
Since $Q$ is an orthogonal matrix, it must be
$$c^2 k=\|c\mathbf{v}\|^2=\|Q\mathbf{1}\|^2=\|\mathbf{1}\|^2=n,$$
therefore $c^2=\frac{n}{k}.$
\\The sum of the degrees of a tree, either $T$ or $S$, is equal to twice the
number of edges, that is, twice the number of vertices minus two. Together
with the fact that $Q$ also conjugates the degree matrices, this yields:
$$2(n-1)=\mathbf{1}^{t}D_T\mathbf{1}
=\mathbf{1}^{t}Q^{t}D_{\Gamma}Q\mathbf{1}
=(Q\mathbf{1})^{t}D_{\Gamma}(Q\mathbf{1})
=c^2\mathbf{v}^{t}D_{\Gamma}\mathbf{v}
=2c^2(k-1),$$
hence $k(n-1)=n(k-1)$, so $n=k$.
\end{proof}
\begin{corollary}
A signed tree is determined by its degree similarity class.
\end{corollary}
\begin{proof}
The result follows  from Lemma~\ref{lem:tree} and Theorem~\ref{theo:tree}.
\end{proof}

In \cite[Problem~27]{GS}, it is asked whether two non-isomorphic unicyclic
graphs can be degree similar: a negative answer has recently
been given in \cite{ZF9oct26}. For unsigned graphs, the question is
equivalent to asking whether unicyclic graphs are determined by their
degree similarity class, since an unsigned graph that is Laplacian
cospectral with a unicyclic graph is itself unicyclic, being connected
with the same number of vertices and edges.
For signed graphs, we do not
know whether two unicyclic  signed graphs which are not switching isomorphic
 can be
degree similar (by Proposition~\ref{prop:sw} and \cite{ZF9oct26}, they would have to be unbalanced); on the other hand, we exhibit pairs of net-degree similar
unicyclic signed graphs with non-isomorphic underlying graphs
(see Examples~\ref{exa:uni} and \ref{exa:rooted}), and unicyclic signed graphs are not, in general,
determined by their degree similarity class, as the following example
shows.
\begin{example}\label{ex:uni}
Let $C_n^+$ and $C_n^-$ denote the balanced and unbalanced signed cycle of
order $n$, respectively. It is easy to see that, for every $n\geq 3$, the
adjacency matrix of $C_{2n}^+$ is cospectral with the adjacency matrix of
$C_n^+\cup C_n^-$ (see, for instance, \cite[Remark~4.5]{BP}; for a complete
classification of the signed graphs that are cospectral with signed cycles,
see \cite{AB}). Since the underlying graphs are $2$-regular, it follows that
$C_{2n}^+$ and $C_n^+\cup C_n^-$ are degree similar for every $n\geq 3$.
Hence, $C_{2n}^+$ is a unicyclic signed graph that is not determined by its
degree similarity class.
\end{example}

We conclude  this section by proving some properties of (net-)degree similarity that
are the signed analogue of \cite[Lemma~5]{GS} and of some results in \cite{FX}, and which will be
useful in the next sections.
\begin{proposition}\label{prop:-1}
Let $\Gamma_1$ and $\Gamma_2$ be two signed graphs, and let
$A_{\Gamma_1}$ and $A_{\Gamma_2}$ denote their adjacency matrices
with respect to an ordering of the vertices in nondecreasing order
of (net-)degree. Then $\Gamma_1$ and $\Gamma_2$ are (net-)degree similar if and only if
there exists an invertible block diagonal matrix
$M=\diag(M_1,\ldots,M_k)$ such that
$$
M^{-1}A_{\Gamma_1}M=A_{\Gamma_2}
$$
and the vertices corresponding to the same block have the same
(net-)degree.
\end{proposition}
\begin{proof}
Suppose that  $\Gamma_1$ and $\Gamma_2$ are (net-)degree similar.
Since two diagonal matrices that are similar have the same  diagonal entries, and both $\Gamma_1$ and
$\Gamma_2$ are indexed in nondecreasing order of (net-)degree, it follows that
$D^{(\pm)}_{\Gamma_1}=D^{(\pm)}_{\Gamma_2}=D^{(\pm)}=\diag(d^{(\pm)}_1 I_{n_1},\ldots,d^{(\pm)}_k I_{n_k})$ where $d^{(\pm)}_1<\cdots<d^{(\pm)}_k$.

By definition there exists an invertible
matrix $M$ such that
$$
M^{-1}A_{\Gamma_1}M = A_{\Gamma_2}
\quad \text{and} \quad
M^{-1}D^{(\pm)}M = D^{(\pm)}.
$$
The result follows by observing that a matrix
$M$ commutes with a diagonal matrix $D^{(\pm)}$ if and only if $M_{i,j}=0$ whenever $D^{(\pm)}_{i,i}\neq D^{(\pm)}_{j,j}$.
Equivalently, $M$ must have the block structure described above.

For the converse implication, suppose that there exists a matrix $M$
with the prescribed block structure, i.e. $M=\diag(M_1,\ldots,M_k)$ with $M_j$ of size $n_j$, such that
$M^{-1}A_{\Gamma_1}M=A_{\Gamma_2}$. Clearly, it also satisfies
$$
M^{-1} D^{(\pm)} M
=
\diag(M_1^{-1},\ldots,M_k^{-1})\,
\diag(d_1^{(\pm)} I_{n_1},\ldots,d_k^{(\pm)} I_{n_k})\,
\diag(M_1,\ldots,M_k)
= D^{(\pm)}.
$$
It follows
that $\Gamma_1$ and $\Gamma_2$ are (net-)degree similar.
\end{proof}

The sufficiency part of Proposition~\ref{prop:-1} does not use the ordering
by (net-)degree, but only that the same ordering is used in both graphs.
We state this version explicitly, as it will be applied with orderings
dictated by vertex partitions.
\begin{proposition}\label{prop:0}
Let $\Gamma_1$ and $\Gamma_2$ be two signed graphs with the same (net-)degree sequence with an ordering such that
$D^{(\pm)}_{\Gamma_1}=D^{(\pm)}_{\Gamma_2}$.
Then $\Gamma_1$ and $\Gamma_2$ are (net-)degree similar if there exists an invertible block diagonal matrix
$M=\diag(M_1,\ldots,M_k)$ such that
$$
M^{-1}A_{\Gamma_1}M=A_{\Gamma_2},
$$
and the vertices corresponding to the same block have the same
(net-)degree.
\end{proposition}

\begin{remark}\label{rem:delete}
If two signed graphs are (net-)degree similar, then deleting all vertices
with a fixed (net-)degree from both graphs produces a pair of cospectral
signed graphs, since the conjugating matrix is block diagonal with
respect to the (net-)degree classes. In particular, in the degree case,
deleting all isolated vertices from both graphs produces a pair of
degree similar signed graphs, by Proposition~\ref{prop:0}, since the
degrees of the remaining vertices are unchanged.
\end{remark}

\section{Local switchings}\label{sec:4}
In the unsigned setting, the local switching introduced in \cite{GM} does not,
in general, preserve degree similarity; a construction based on it which does
was given in \cite{GS}. In this section we develop a similar construction for
signed graphs, starting from the local switching introduced in \cite{cos};
 even when restricted to unsigned graphs, it strictly generalizes the one
in \cite{GS} (see Subsection~\ref{sub:4.4}).
\\Consider a signed graph $\Gamma=(G,\sigma)$, with $G=(V_G,E_G)$.
Let $\pi=(C_0,C_1,\ldots,C_p)$ be a partition of the set $V_G$, such that $|C_i|=n_i$, for $i=0,\ldots,p$.
For $v \in V_G$, we denote by $d_i^+(v)$ and $d_i^-(v)$ the number of
positive and negative neighbors of $v$ in $C_i$, respectively, and set
$
d_i^\pm(v) := d_i^+(v) - d_i^-(v)$ and $d_i(v):=d_i^+(v) + d_i^-(v).
$

\subsection{Preserving net-degree similarity}\label{sec:4.1}
A partition $\pi = (C_0, C_1, \ldots, C_p)$ of $V_\Gamma$
is called a \emph{net-degree-preserving switching partition}
if the following conditions are satisfied.
\begin{enumerate}[i)]
\item For every $i = 1,2,\ldots,p$, for every $j =0,1,\ldots,p$  there exist integers $d^\pm_{i,j} \in \mathbb{Z}$
such that
$
d^\pm_{j}(v)=d^\pm_{i,j},
$
for every vertex $v \in C_i$.
\item For every $i = 1,2,\ldots,p$ and every $v \in C_0$  one of the following conditions holds:
\begin{enumerate}[(1)]
\item $d^\pm_i(v)=0$;
\item $(d^+_i(v),d^-_i(v))\in\{(n_i,0),(0,n_i), (n_i/2,0),(0,n_i/2)\}.$
\end{enumerate}
\end{enumerate}

If $\pi$ is a net-degree-preserving switching partition, we construct a new
signed graph $\Gamma^\pi$  as in the switching defined
in \cite{cos}. \\
When $v\in C_0$ and
$i \in \{1,2,\ldots,p\}$ are such that  ii) is satisfied via (1),
we switch the sign of every edge between $v$ and $C_i$.
When $v\in C_0$ and
$i \in \{1,2,\ldots,p\}$ are such that  ii) is satisfied via (2) and exactly half of the vertices of
$C_i$ are adjacent to $v$ (either all positively or all negatively), we delete
all edges between $v$ and $C_i$ and join $v$ to the vertices of
$C_i$ that were not adjacent to  $v$, with the same sign.
All other adjacencies remain unchanged.

\begin{theorem}\label{prop:ns}
The signed graphs $\Gamma$ and $\Gamma^\pi$ are net-degree similar.
\end{theorem}
\begin{proof}
Conditions i) and ii) imply the corresponding assumptions of
\cite[Theorem 3.2]{cos}, so that $A_{\Gamma^\pi}=Q^{-1}A_\Gamma Q$, where
$Q=\diag(I_{n_0},Q_{n_1},\ldots,Q_{n_p})$ and $Q_{n_i}=\frac{2}{n_i}J-I$. Observe that $Q^{-1}=Q$.
\\ By condition i)  one has, for every vertex $v\in C_i$ in $\Gamma$, with $i=1,\ldots,p$, that
\begin{equation*}\label{eq:sommagrado}
d^\pm(v)
=
\sum_{j=0}^{p} d_j^\pm(v)
=
\sum_{j=0}^{p} d_{i,j}^\pm.
\end{equation*}
Therefore the net-degree vector
$\mathbf d_\Gamma^\pm$ is constant on each $C_i$, for $i=1,\ldots,p$.
Since $Q$ fixes every vector with this property (see \cite[Lemma 3.1]{cos}), we have
$Q\mathbf 1=\mathbf 1$ and $Q\mathbf d_\Gamma^\pm=\mathbf d_\Gamma^\pm$, hence
$$
\mathbf d^\pm_{\Gamma^{\pi}}=A_{\Gamma^\pi}\mathbf 1=Q^{-1}A_\Gamma Q\mathbf 1=Q^{-1}A_\Gamma\mathbf 1
=Q^{-1}\mathbf d_\Gamma^\pm=Q\mathbf d_\Gamma^\pm=\mathbf d_\Gamma^\pm .
$$
The claim follows from Proposition~\ref{prop:0}, since
$D^\pm_\Gamma=D^\pm_{\Gamma^\pi}$ and the vertices corresponding to the same
block in
$
Q=\mathrm{diag}(\underbrace{1,\ldots,1}_{n_0},Q_{n_1},\ldots,Q_{n_p})
$
have the same net-degree.
\end{proof}

\begin{example}\label{exa:uni}
Consider the unicyclic signed graph $\Gamma$ on the left of
Fig.~\ref{fig:uni2}, and the partition $\pi=(C_0,C_1)$ of $V_\Gamma$ with
$$
C_0=\{1,2,3,4,5\},\qquad C_1=\{6,7,8,9\}.
$$
We have $d^\pm_{1,1}=0$ and $d^\pm_{1,0}=1$ so that condition i)
holds.  Moreover $d^\pm_1(1)=d^\pm_1(2)=d^\pm_1(3)=0$, whereas
$(d^+_1(4),d^-_1(4))=(d^+_1(5),d^-_1(5))=(2,0)=(n_1/2,0)$, so also condition
ii) holds. The signed graph $\Gamma^\pi$ is depicted on the right of
Fig.~\ref{fig:uni2}; observe that its underlying graph is not isomorphic
to that of $\Gamma$.
By Theorem~\ref{prop:ns}, the signed graphs $\Gamma$ and $\Gamma^\pi$ are net-degree
similar.  Moreover, $\Gamma$ and $\Gamma^\pi$ are not degree similar, since their
second Zagreb indices  are $47$ and $49$,
respectively (Proposition~\ref{prop:specht}).
\end{example}
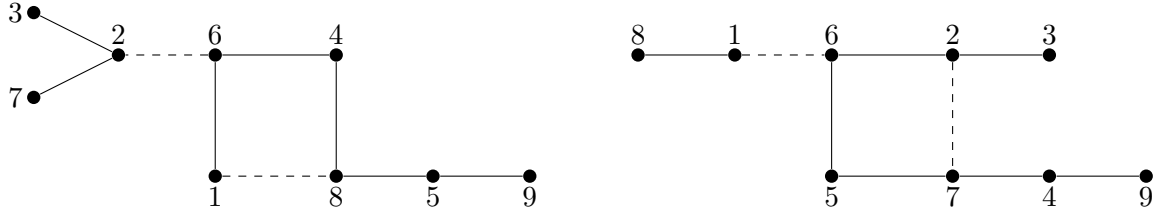
\begin{figure}
\begin{centering}
\begin{tikzpicture}[scale=0.8,
  vertex/.style={circle, fill=black, minimum size=5pt, inner sep=0pt},
  hvertex/.style={circle, draw=black, fill=white, minimum size=5pt, inner sep=0pt, thick},
  pos edge/.style={thin},
  neg edge/.style={thin, dashed},
  lbl/.style={draw=none, fill=none, font=\small}]
\node[vertex]  (1) at (0,0)      {};
\node[vertex] (6) at (0,2)      {};
\node[vertex]  (4) at (2,2)      {};
\node[vertex] (8) at (2,0)      {};
\node[vertex]  (2) at (-1.6,2)   {};
\node[vertex]  (3) at (-3,2.7)   {};
\node[vertex] (7) at (-3,1.3)   {};
\node[vertex]  (5) at (3.6,0)    {};
\node[vertex] (9) at (5.2,0)    {};
\node[lbl,below] at (0,0)     {$1$};
\node[lbl,above] at (0,2)     {$6$};
\node[lbl,above] at (2,2)     {$4$};
\node[lbl,below] at (2,0)     {$8$};
\node[lbl,above] at (-1.6,2)  {$2$};
\node[lbl,left]  at (-3,2.7)  {$3$};
\node[lbl,left]  at (-3,1.3)  {$7$};
\node[lbl,below] at (3.6,0)   {$5$};
\node[lbl,below] at (5.2,0)   {$9$};
\draw[pos edge] (1) -- (6);
\draw[neg edge] (1) -- (8);
\draw[neg edge] (6) -- (2);
\draw[pos edge] (6) -- (4);
\draw[pos edge] (2) -- (7);
\draw[pos edge] (2) -- (3);
\draw[pos edge] (4) -- (8);
\draw[pos edge] (8) -- (5);
\draw[pos edge] (5) -- (9);
\end{tikzpicture}
\qquad
\begin{tikzpicture}[scale=0.8,
  vertex/.style={circle, fill=black, minimum size=5pt, inner sep=0pt},
  hvertex/.style={circle, draw=black, fill=white, minimum size=5pt, inner sep=0pt, thick},
  pos edge/.style={thin},
  neg edge/.style={thin, dashed},
  lbl/.style={draw=none, fill=none, font=\small}]
\node[vertex]  (5) at (0,0)      {};
\node[vertex] (6) at (0,2)      {};
\node[vertex]  (2) at (2,2)      {};
\node[vertex] (7) at (2,0)      {};
\node[vertex]  (1) at (-1.6,2)   {};
\node[vertex] (8) at (-3.2,2)   {};
\node[vertex]  (3) at (3.6,2)    {};
\node[vertex]  (4) at (3.6,0)    {};
\node[vertex] (9) at (5.2,0)    {};
\node[lbl,below] at (0,0)     {$5$};
\node[lbl,above] at (0,2)     {$6$};
\node[lbl,above] at (2,2)     {$2$};
\node[lbl,below] at (2,0)     {$7$};
\node[lbl,above] at (-1.6,2)  {$1$};
\node[lbl,above] at (-3.2,2)  {$8$};
\node[lbl,above] at (3.6,2)   {$3$};
\node[lbl,below] at (3.6,0)   {$4$};
\node[lbl,below] at (5.2,0)   {$9$};
\draw[neg edge] (1) -- (6);
\draw[pos edge] (1) -- (8);
\draw[pos edge] (2) -- (6);
\draw[neg edge] (2) -- (7);
\draw[pos edge] (2) -- (3);
\draw[pos edge] (4) -- (7);
\draw[pos edge] (4) -- (9);
\draw[pos edge] (5) -- (6);
\draw[pos edge] (5) -- (7);
\end{tikzpicture}
\end{centering}
\caption{The net-degree similar unicyclic signed graphs of Example~\ref{exa:uni}.}\label{fig:uni2}
\end{figure}

\subsection{Preserving degree similarity}\label{sec:4.2}
A partition $\pi = (C_0, C_1, \ldots, C_p)$ of $V_\Gamma$
is called a \emph{degree-preserving switching partition}
if the following conditions are satisfied.
\begin{enumerate}[i)]
\item For every $i,j =1,2,\ldots,p$  there exist integers $d^+_{i,j},d^-_{i,j}\in \mathbb{Z}$
such that $(d^+_j(v),d^-_j(v))=(d^+_{i,j},d^-_{i,j})$
for every vertex $v \in C_i$.
\item For every $j = 1,2,\ldots,p$, one of the following conditions holds:
\begin{enumerate}[(1)]
\item $d^\pm_j(v)=0$ for every $v \in C_0$;
\item $(d^+_j(v),d^-_j(v))\in\{(n_j,0),(0,n_j), (n_j/2,0),(0,n_j/2),(n_j/2,n_j/2), (0,0)\}$ for every $v \in C_0$.
\end{enumerate}
\item  For every $i = 1,2,\ldots,p$ there exists an integer $d_{i,0}\geq 0$ such that
$
d_{0}(v)=d_{i,0},
$
for every vertex $v \in C_i$.
\end{enumerate}

If $\pi$ is a degree-preserving switching partition, we define $\Gamma^\pi$
exactly as in Section~\ref{sec:4.1}, that is,
 for every $v\in C_0$ and every $j\in\{1,\ldots,p\}$,
if $d^\pm_j(v)=0$, we switch the sign of every edge between $v$ and $C_j$;
if $(d^+_j(v),d^-_j(v))\in\{(n_j/2,0),(0,n_j/2)\}$, we delete all edges
between $v$ and $C_j$ and join $v$ to the $n_j/2$ vertices of $C_j$ that
were not adjacent to $v$, with the same sign. All other adjacencies remain
unchanged.

Notice that the argument used in the proof of
Theorem~\ref{prop:ns} in the net-degree case relies on the identity $\mathbf d^\pm_\Gamma
= A_\Gamma\mathbf 1$; the degree vector admits no such expression in terms
of $A_\Gamma$, hence the conjugation $A_{\Gamma^\pi}=Q^{-1}A_\Gamma Q$
gives no direct information on it, and the degrees must be tracked along
the switching. Moreover, condition ii) is now required to hold uniformly on
$C_0$ for each $j$.

\begin{theorem}\label{prop:ds}
The signed graphs $\Gamma$ and $\Gamma^\pi$ are degree similar.
\end{theorem}
\begin{proof}
Conditions i) and ii) imply the assumptions of \cite[Theorem 3.2]{cos}, so that $A_{\Gamma^\pi}=Q^{-1}A_\Gamma Q$ with the same
matrix $Q=\diag(I_{n_0},Q_{n_1},\ldots,Q_{n_p})$.
For every $v\in C_i$, with $i=1,\ldots,p$, conditions i) and iii) imply
\begin{equation*}\label{eq:sommagrado2}
d(v)=d_0(v)+\sum_{j=1}^{p}\big(d_j^+(v)+d_j^-(v)\big)
=d_{i,0}+\sum_{j=1}^{p}\big(d^+_{i,j}+d^-_{i,j}\big),
\end{equation*}
so that  all vertices in $C_i$, for $i=1,\ldots, p$, have the same
degree in $\Gamma$. Since
$$
Q=\mathrm{diag}(\underbrace{1,\ldots,1}_{n_0},Q_{n_1},\ldots,Q_{n_p}),
$$
by Proposition~\ref{prop:0} the claim follows once we prove that
$D_\Gamma=D_{\Gamma^\pi}$.
Let $v\in C_0$. Each operation performed on the edges between $v$
and some $C_j$ either changes signs only, or replaces the $n_j/2$ edges
from $v$ to $C_j$ with $n_j/2$ edges from $v$ to $C_j$; hence the degree
of $v$ is unchanged.
\\ Now let $w\in C_j$, with $j\in\{1,\ldots,p\}$. The adjacencies of $w$
with vertices of $C_1\cup\cdots\cup C_p$ are unchanged, so we only need
to look at the edges between $w$ and $C_0$. If $d^\pm_j(v)=0$ for every
$v\in C_0$ (condition ii) holds via (1)) then only sign changes occur between $C_0$ and $C_j$, and the
degree of $w$ is preserved. Otherwise condition ii) holds via (2), and
$C_0$ is the disjoint union of
\[
F_j=\{v\in C_0: d_j(v)=n_j\},\quad
H_j=\{v\in C_0: d_j(v)=n_j/2\},\quad
Z_j=\{v\in C_0: d_j(v)=0\}.
\]
Every $w\in C_j$ is adjacent to all vertices of $F_j$ and to no vertex of
$Z_j$, hence by condition iii)
\[
|F_j|+|H_j\cap N(w)|=d_0(w)=d_{j,0}\qquad\text{for every } w\in C_j .
\]
Summing over $w\in C_j$ and counting the edges between $C_0$ and $C_j$, we get
$$n_j|F_j|+\tfrac{n_j}{2}|H_j|=n_jd_{j,0},$$ that is
$d_{j,0}-|F_j|=|H_j|/2$. Therefore
\[
|H_j\cap N(w)|=\tfrac12|H_j|=|H_j\setminus N(w)|
\qquad\text{for every } w\in C_j .
\]
In $\Gamma^\pi$ the vertex $w$ is still adjacent to all vertices of
$F_j$ (possibly with different signs), to no vertex of $Z_j$, and to the
vertices of $H_j\setminus N(w)$ instead of those of $H_j\cap N(w)$.
Hence the degree of $w$ is preserved, and $D_\Gamma=D_{\Gamma^\pi}$.
\end{proof}

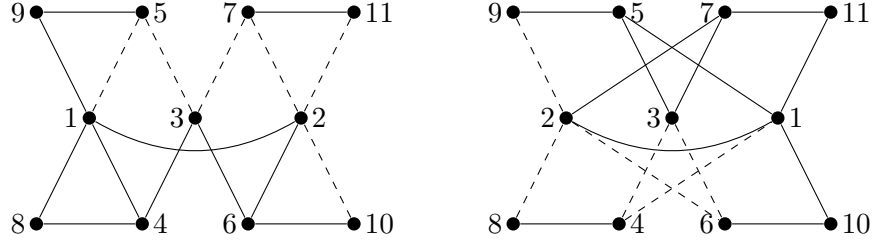
\begin{figure}
\begin{centering}
\begin{tikzpicture}[scale=0.7,
  vertex/.style={circle, fill=black, minimum size=5pt, inner sep=0pt},
  pos edge/.style={thin},
  neg edge/.style={thin, dashed},
  lbl/.style={draw=none, fill=none, font=\small}
]
\node[vertex] (11) at (6, 4) {};
\node[vertex] (10) at (6, 0) {};
\node[vertex] (7)  at (4, 4) {};
\node[vertex] (2)  at (5, 2) {};
\node[vertex] (6)  at (4, 0) {};
\node[vertex] (3)  at (3, 2) {};
\node[vertex] (5)  at (2, 4) {};
\node[vertex] (9)  at (0, 4) {};
\node[vertex] (1)  at (1, 2) {};
\node[vertex] (8)  at (0, 0) {};
\node[vertex] (4)  at (2, 0) {};
\node[lbl,right] at (6,4) {$11$};
\node[lbl,right] at (6,0) {$10$};
\node[lbl,left]  at (4,4) {$7$};
\node[lbl,right] at (5,2) {$2$};
\node[lbl,left]  at (4,0) {$6$};
\node[lbl,left]  at (3,2) {$3$};
\node[lbl,right] at (2,4) {$5$};
\node[lbl,left]  at (0,4) {$9$};
\node[lbl,left]  at (1,2) {$1$};
\node[lbl,left]  at (0,0) {$8$};
\node[lbl,right] at (2,0) {$4$};
\draw[pos edge] (11) -- (7);
\draw[pos edge] (2) -- (6);
\draw[pos edge] (10) -- (6);
\draw[neg edge] (11) -- (2);
\draw[neg edge] (10) -- (2);
\draw[neg edge] (7) -- (2);
\draw[pos edge] (9) -- (1);
\draw[pos edge] (8) -- (1);
\draw[pos edge] (4) -- (1);
\draw[pos edge] (4) -- (8);
\draw[pos edge] (5) -- (9);
\draw[pos edge] (3) -- (4);
\draw[pos edge] (3) -- (6);
\draw[neg edge] (5) -- (1);
\draw[neg edge] (3) -- (5);
\draw[neg edge] (3) -- (7);
\draw[pos edge] (2) to [bend left=30] (1);
\end{tikzpicture}
\qquad
\begin{tikzpicture}[scale=0.7,
  vertex/.style={circle, fill=black, minimum size=5pt, inner sep=0pt},
  pos edge/.style={thin},
  neg edge/.style={thin, dashed},
  lbl/.style={draw=none, fill=none, font=\small}
]
\node[vertex] (11) at (6, 4) {};
\node[vertex] (10) at (6, 0) {};
\node[vertex] (7)  at (4, 4) {};
\node[vertex] (2)  at (1, 2) {};
\node[vertex] (6)  at (4, 0) {};
\node[vertex] (3)  at (3, 2) {};
\node[vertex] (5)  at (2, 4) {};
\node[vertex] (9)  at (0, 4) {};
\node[vertex] (1)  at (5, 2) {};
\node[vertex] (8)  at (0, 0) {};
\node[vertex] (4)  at (2, 0) {};
\node[lbl,right] at (6,4) {$11$};
\node[lbl,right] at (6,0) {$10$};
\node[lbl,left]  at (4,4) {$7$};
\node[lbl,right] at (5,2) {$1$};
\node[lbl,left]  at (4,0) {$6$};
\node[lbl,left]  at (3,2) {$3$};
\node[lbl,right] at (2,4) {$5$};
\node[lbl,left]  at (0,4) {$9$};
\node[lbl,left]  at (1,2) {$2$};
\node[lbl,left]  at (0,0) {$8$};
\node[lbl,right] at (2,0) {$4$};
\draw[pos edge] (11) -- (7);
\draw[neg edge] (2) -- (6);
\draw[pos edge] (10) -- (6);
\draw[neg edge] (9) -- (2);
\draw[neg edge] (8) -- (2);
\draw[pos edge] (7) -- (2);
\draw[pos edge] (11) -- (1);
\draw[pos edge] (10) -- (1);
\draw[neg edge] (4) -- (1);
\draw[pos edge] (4) -- (8);
\draw[pos edge] (5) -- (9);
\draw[neg edge] (3) -- (4);
\draw[neg edge] (3) -- (6);
\draw[pos edge] (5) -- (1);
\draw[pos edge] (3) -- (5);
\draw[pos edge] (3) -- (7);
\draw[pos edge] (2) to [bend right=30] (1);
\end{tikzpicture}
\end{centering}
\caption{The degree similar signed graphs $\Gamma$ and $\Gamma^\pi$ of Example~\ref{exa:2}.}\label{fig:2}
\end{figure}

\begin{example}\label{exa:2}
Consider the signed graph $\Gamma$ on the left of Fig.~\ref{fig:2}, and the
partition $\pi=(C_0,C_1,C_2)$ of its vertex set $V_\Gamma$ such that
$$
C_0=\{1,2,3\},\; C_1=\{4,5,6,7\}, \;C_2=\{8,9,10,11\}.
$$
Condition i) is satisfied with
$(d_{1,1}^+,d_{1,1}^-)=(d_{2,2}^+,d_{2,2}^-)=(0,0)$ and
$(d_{1,2}^+,d_{1,2}^-)=(d_{2,1}^+,d_{2,1}^-)=(1,0)$.
Condition iii) is satisfied with $d_{1,0}=2$ and $d_{2,0}=1$.
Concerning Condition ii), for $j=1$ we have
$d_1^\pm(1)=d_1^\pm(2)=d_1^\pm(3)=0$, so that the point (1) holds. For
$j=2$ observe that $n_2=4$ and
$$
(d_2^+(1),d_2^-(1))=(2,0), \quad (d_2^+(2),d_2^-(2))=(0,2), \quad
(d_2^+(3),d_2^-(3))=(0,0),
$$
so that the point (2) holds. By Theorem~\ref{prop:ds}, the signed graph
$\Gamma^\pi$ on the right of Fig.~\ref{fig:2} is degree similar to
$\Gamma$. However, their underlying graphs are not isomorphic, since the
first one contains triangles whereas the second one is triangle-free.
Also, observe that $\Gamma$ and $\Gamma^\pi$ are not net-degree similar by Proposition~\ref{prop:specht}, since $\operatorname{tr}\big(A_\Gamma D^\pm_\Gamma A_\Gamma^2 D^\pm_\Gamma\big)=48$ whereas $\operatorname{tr}\big(A_{\Gamma^\pi} D^\pm_{\Gamma^\pi} A_{\Gamma^\pi}^2 D^\pm_{\Gamma^\pi}\big)=0$.
\end{example}

\subsection{Preserving net-degree and degree similarity}
It follows from Sections \ref{sec:4.1} and \ref{sec:4.2}  that a
 partition $\pi = (C_0, C_1, \ldots, C_p)$ of $V_\Gamma$
is both net-degree-preserving  and degree-preserving switching if the following conditions are satisfied.
\begin{enumerate}[i)]
\item For every $i =1,2,\ldots,p$, for every $j=0,1,2,\ldots,p$  there exist integers $d^+_{i,j},d^-_{i,j}\in \mathbb{Z}$
such that $(d^+_j(v),d^-_j(v))=(d^+_{i,j},d^-_{i,j})$
for every vertex $v \in C_i$.
\item For every $j = 1,2,\ldots,p$, one of the following conditions holds:
\begin{enumerate}[(1)]
\item $d^\pm_j(v)=0$ for every $v \in C_0$;
\item $(d^+_j(v),d^-_j(v))\in\{(n_j,0),(0,n_j), (n_j/2,0),(0,n_j/2),(n_j/2,n_j/2), (0,0)\}$ for every $v \in C_0$.
\end{enumerate}
\end{enumerate}
It then follows from Theorems~\ref{prop:ns} and~\ref{prop:ds} that, for
any such partition $\pi$, the signed graphs $\Gamma$ and $\Gamma^\pi$
are simultaneously net-degree similar and degree similar, with the same
conjugating matrix.

\begin{example}\label{exa:3}
Consider the signed graph $\Gamma$ on the left of Fig.~\ref{fig:3}, and the partition $\pi=(C_0,C_1)$ of its vertex set $V_\Gamma$ such that
$$
C_0=\{1,2,3\},\; C_1=\{4,5,6,7,8,9\}.
$$
We have:
$
(d^+_{1,0},d^-_{1,0})=(1,1)$, $(d^+_{1,1},d^-_{1,1})=(2,1)$,
$
(d_1^+(1),d_1^-(1))=(d_1^+(2),d_1^-(2))=(3,0)$ and  $(d_1^+(3),d_1^-(3))=(0,6).$
In particular, the graph $\Gamma^\pi$ on the right of Fig.~\ref{fig:3} is both net-degree similar and degree similar with the signed graph $\Gamma$.
Observe that the underlying graphs of $\Gamma$ and $\Gamma^\pi$ are isomorphic. A direct computation shows that  the multisets of the diagonal
entries of the fourth powers of the adjacency matrices of the two signed
graphs are different. Each such entry is the difference between the numbers
of positive and negative closed walks of length four based at the
corresponding vertex. Hence, the multiset of these values over all vertices
is an invariant under switching isomorphism, and the two signed graphs are
not switching isomorphic.
\end{example}

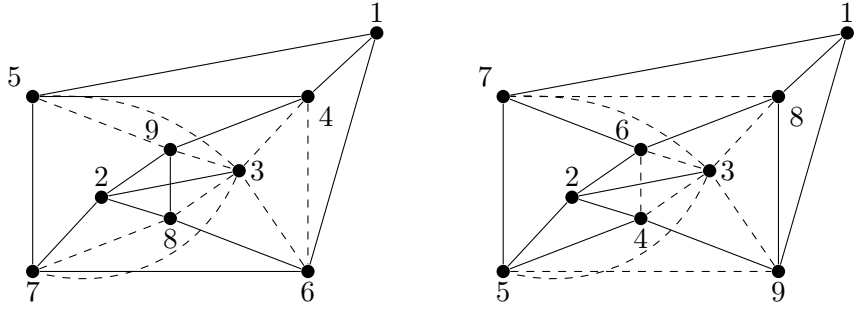
\begin{figure}
\begin{centering}
\begin{tikzpicture}[scale=0.7,
  vertex/.style={circle, fill=black, minimum size=5pt, inner sep=0pt},
  pos edge/.style={thin},
  neg edge/.style={thin, dashed},
  lbl/.style={draw=none, fill=none, font=\small}
]
\node[vertex] (4) at (0.00, 0.0) {};
\node[vertex] (3) at (5.20, 0.0) {};
\node[vertex] (1) at (5.20, 3.3) {};
\node[vertex] (2) at (0.00, 3.3) {};
\node[vertex] (5) at (2.60, 1.0) {};
\node[vertex] (C) at (3.90, 1.9) {};
\node[vertex] (6) at (2.60, 2.3) {};
\node[vertex] (B) at (1.30, 1.4) {};
\node[vertex] (A) at (6.50, 4.5) {};
\node[lbl,below] at (0.00, 0.0)  {$7$};
\node[lbl,below] at (5.20, 0.0)  {$6$};
\node[lbl,below right] at (5.20, 3.3) {$4$};
\node[lbl,above left] at (0.00, 3.3) {$5$};
\node[lbl,right] at (3.90, 1.9) {$3$};
\node[lbl,below] at (2.60, 1.0) {$8$};
\node[lbl,above] at (1.30, 1.4) {$2$};
\node[lbl,above left] at (2.60, 2.3) {$9$};
\node[lbl,above] at (6.50, 4.5) {$1$};
\draw[pos edge] (1) -- (2);
\draw[pos edge] (2) -- (4);
\draw[pos edge] (4) -- (3);
\draw[neg edge] (1) -- (3);
\draw[pos edge] (A) -- (2);
\draw[pos edge] (A) -- (1);
\draw[pos edge] (A) -- (3);
\draw[pos edge] (6) -- (5);
\draw[pos edge] (B) -- (5);
\draw[pos edge] (B) -- (6);
\draw[pos edge] (B) -- (C);
\draw[neg edge] (6) -- (C);
\draw[neg edge] (C) -- (5);
\draw[neg edge] (1) -- (C);
\draw[pos edge] (1) -- (6);
\draw[neg edge] (2) -- (6);
\draw[pos edge] (B) -- (4);
\draw[neg edge] (4) -- (5);
\draw[pos edge] (5) -- (3);
\draw[neg edge] (3) -- (C);
\draw[neg edge] (C) to [bend right=22] (2);
\draw[neg edge] (C) to [bend left=40] (4);
\end{tikzpicture}
\qquad
\begin{tikzpicture}[scale=0.7,
  vertex/.style={circle, fill=black, minimum size=5pt, inner sep=0pt},
  pos edge/.style={thin},
  neg edge/.style={thin, dashed},
  lbl/.style={draw=none, fill=none, font=\small}
]
\node[vertex] (4) at (0.00, 0.0) {};
\node[vertex] (3) at (5.20, 0.0) {};
\node[vertex] (1) at (5.20, 3.3) {};
\node[vertex] (2) at (0.00, 3.3) {};
\node[vertex] (5) at (2.60, 1.0) {};
\node[vertex] (C) at (3.90, 1.9) {};
\node[vertex] (6) at (2.60, 2.3) {};
\node[vertex] (B) at (1.30, 1.4) {};
\node[vertex] (A) at (6.50, 4.5) {};
\node[lbl,below] at (0.00, 0.0)  {$5$};
\node[lbl,below] at (5.20, 0.0)  {$9$};
\node[lbl,below right] at (5.20, 3.3) {$8$};
\node[lbl,above left] at (0.00, 3.3) {$7$};
\node[lbl,right] at (3.90, 1.9) {$3$};
\node[lbl,below] at (2.60, 1.0) {$4$};
\node[lbl,above] at (1.30, 1.4) {$2$};
\node[lbl,above left] at (2.60, 2.3) {$6$};
\node[lbl,above] at (6.50, 4.5) {$1$};
\draw[neg edge] (1) -- (2);
\draw[pos edge] (2) -- (4);
\draw[neg edge] (4) -- (3);
\draw[pos edge] (1) -- (3);
\draw[pos edge] (A) -- (2);
\draw[pos edge] (A) -- (1);
\draw[pos edge] (A) -- (3);
\draw[neg edge] (6) -- (5);
\draw[pos edge] (B) -- (5);
\draw[pos edge] (B) -- (6);
\draw[pos edge] (B) -- (C);
\draw[neg edge] (6) -- (C);
\draw[neg edge] (C) -- (5);
\draw[neg edge] (1) -- (C);
\draw[pos edge] (1) -- (6);
\draw[pos edge] (2) -- (6);
\draw[pos edge] (B) -- (4);
\draw[pos edge] (4) -- (5);
\draw[pos edge] (5) -- (3);
\draw[neg edge] (3) -- (C);
\draw[neg edge] (C) to [bend right=22] (2);
\draw[neg edge] (C) to [bend left=40] (4);
\end{tikzpicture}
\end{centering}
\caption{The net-degree similar and degree similar signed graphs $\Gamma$ and $\Gamma^\pi$ of Example~\ref{exa:3}.}\label{fig:3}
\end{figure}

\subsection{All positive (all negative) case}\label{sub:4.4}
A special case of partitions that are simultaneously net-degree-preserving
and degree-preserving switching occurs when the graph is entirely positive
(or entirely negative). In this case each of the two definitions is
equivalent to the simultaneous satisfaction of the following two
conditions:
\begin{enumerate}[i)]
\item For every $i = 1,2,\ldots,p$, for every $j =0,1,\ldots,p$ there exist
integers $d_{i,j} \in \mathbb{Z}$ such that $d_{j}(v)=d_{i,j}$ for every
vertex $v \in C_i$.
\item For every $i = 1,2,\ldots,p$ and every $v \in C_0$ one has
$d_i(v)\in \{0, n_i/2, n_i\}$.
\end{enumerate}
These are therefore conditions for switching of unsigned graphs ensuring
that the switched graph is degree similar to the original one, where
degree similarity can be understood either in the sense of
Definition~\ref{def:prima} or in the sense of \cite{GS}.

Conditions~i) and~ii) extend the GM$^*$ property introduced by Haemers and
Spence~\cite{HS} from partitions into two parts to partitions with several
parts. In~\cite{HS} it is
observed that GM$^*$ switching produces graphs that are cospectral with
respect to every matrix of the form $\alpha A_G+\beta I+\delta D_G$. It follows immediately from their
argument that such graphs are in fact degree similar. In particular, the column GM$^*$ of
\cite[Table~1]{HS} gives lower bounds for the number of graphs on at most
$11$ vertices that are not determined by their degree similarity class.
The conditions of \cite[Lemma~9]{GS} are a special case of i) and ii), in
which every vertex of $C_0$ has exactly $n_i/2$ neighbours in each $C_i$,
and every vertex of $C_i$ has exactly $|C_0|/2$ neighbours in $C_0$.
Example~\ref{exa:4} shows that this inclusion is strict.

\begin{example}\label{exa:4}
Consider the all-positive signed graph $G$ on the left of
Fig.~\ref{fig:4}, and the partition $\pi=(C_0,C_1)$ of its vertex set with
$C_0=\{1,2,3\}$ and $C_1=\{4,5,\ldots,11\}$. Then $d_{1,0}=1$,
$d_{1,1}=3$, $d_1(1)=d_1(2)=4=|C_1|/2$ and $d_1(3)=0$, so that conditions
i) and ii) are satisfied, whereas those of \cite[Lemma~9]{GS} are not,
since vertex $3$ has no neighbours in $C_1$. The switched graph $G^\pi$ is depicted on the right of Fig.~\ref{fig:4}. The graphs $G$ and $G^\pi$ are degree similar but not
isomorphic, since in the cubic subgraph induced by $C_1$ the subsets
$\{4,5,6,7\}$ and $\{8,9,10,11\}$ do not belong to the same orbit of the
automorphism group.
\end{example}

\begin{figure}
\begin{centering}
\begin{tikzpicture}[scale=0.75,
  vertex/.style={circle, fill=black, minimum size=5pt, inner sep=0pt},
  pos edge/.style={thin},
  neg edge/.style={thin, dashed},
  lbl/.style={draw=none, fill=none, font=\small}
]
\node[vertex] (1)  at (0, 1.5) {};
\node[vertex] (10) at (1, 3)   {};
\node[vertex] (9)  at (1, 1)   {};
\node[vertex] (8)  at (2, 0)   {};
\node[vertex] (11) at (2, 2)   {};
\node[vertex] (6)  at (4, 3)   {};
\node[vertex] (5)  at (4, 1)   {};
\node[vertex] (7)  at (5, 2)   {};
\node[vertex] (4)  at (5, 0)   {};
\node[vertex] (2)  at (6, 1.5) {};
\node[vertex] (3)  at (7, 1.5) {};
\node[lbl,left]        at (0, 1.5) {$1$};
\node[lbl,above right] at (1, 3)   {$10$};
\node[lbl,above right] at (1, 1)   {$9$};
\node[lbl,above right] at (2, 0)   {$8$};
\node[lbl,above]       at (2, 2)   {$11$};
\node[lbl,above right] at (4, 3)   {$6$};
\node[lbl,above right] at (4, 1)   {$5$};
\node[lbl,above right] at (5, 2)   {$7$};
\node[lbl,right]       at (5, 0)   {$4$};
\node[lbl,below right] at (6, 1.5) {$2$};
\node[lbl,below]       at (7, 1.5) {$3$};
\draw[pos edge] (1) -- (9);
\draw[pos edge] (1) -- (8);
\draw[pos edge] (1) -- (10);
\draw[pos edge] (1) -- (11);
\draw[pos edge] (9) -- (8);
\draw[pos edge] (9) -- (10);
\draw[pos edge] (8) -- (10);
\draw[pos edge] (10) -- (11);
\draw[pos edge] (11) -- (6);
\draw[pos edge] (7) -- (11);
\draw[pos edge] (5) -- (9);
\draw[pos edge] (4) -- (8);
\draw[pos edge] (4) -- (5);
\draw[pos edge] (4) -- (7);
\draw[pos edge] (6) -- (7);
\draw[pos edge] (6) -- (5);
\draw[pos edge] (6) -- (2);
\draw[pos edge] (5) -- (2);
\draw[pos edge] (4) -- (2);
\draw[pos edge] (7) -- (2);
\draw[pos edge] (3) -- (2);
\end{tikzpicture}
\qquad
\begin{tikzpicture}[scale=0.75,
  vertex/.style={circle, fill=black, minimum size=5pt, inner sep=0pt},
  pos edge/.style={thin},
  neg edge/.style={thin, dashed},
  lbl/.style={draw=none, fill=none, font=\small}
]
\node[vertex] (1)  at (0, 1.5)  {};
\node[vertex] (10) at (1, 3)    {};
\node[vertex] (9)  at (1, 1)    {};
\node[vertex] (8)  at (2, 0)    {};
\node[vertex] (11) at (2, 2)    {};
\node[vertex] (6)  at (4, 3)    {};
\node[vertex] (5)  at (4, 1)    {};
\node[vertex] (7)  at (5, 2)    {};
\node[vertex] (4)  at (5, 0)    {};
\node[vertex] (2)  at (6, 1.5)  {};
\node[vertex] (3)  at (-1, 1.5) {};
\node[lbl,below]       at (0, 1.5)  {$2$};
\node[lbl,above right] at (1, 3)    {$10$};
\node[lbl,above right] at (1, 1)    {$9$};
\node[lbl,above right] at (2, 0)    {$8$};
\node[lbl,above]       at (2, 2)    {$11$};
\node[lbl,above right] at (4, 3)    {$6$};
\node[lbl,above right] at (4, 1)    {$5$};
\node[lbl,above right] at (5, 2)    {$7$};
\node[lbl,right]       at (5, 0)    {$4$};
\node[lbl,below right] at (6, 1.5)  {$1$};
\node[lbl,below]       at (-1, 1.5) {$3$};
\draw[pos edge] (1) -- (9);
\draw[pos edge] (1) -- (8);
\draw[pos edge] (1) -- (10);
\draw[pos edge] (1) -- (11);
\draw[pos edge] (9) -- (8);
\draw[pos edge] (9) -- (10);
\draw[pos edge] (8) -- (10);
\draw[pos edge] (10) -- (11);
\draw[pos edge] (11) -- (6);
\draw[pos edge] (7) -- (11);
\draw[pos edge] (5) -- (9);
\draw[pos edge] (4) -- (8);
\draw[pos edge] (4) -- (5);
\draw[pos edge] (4) -- (7);
\draw[pos edge] (6) -- (7);
\draw[pos edge] (6) -- (5);
\draw[pos edge] (6) -- (2);
\draw[pos edge] (5) -- (2);
\draw[pos edge] (4) -- (2);
\draw[pos edge] (7) -- (2);
\draw[pos edge] (3) -- (1);
\end{tikzpicture}
\end{centering}
\caption{The degree similar graphs $G$ and $G^\pi$ of Example~\ref{exa:4}.}\label{fig:4}
\end{figure}
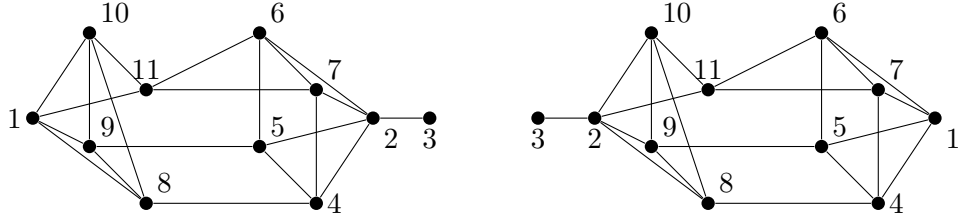

\section{Partial transposition}\label{sec:5}
In this section we employ a switching of a completely different kind to build \mbox{(net-)degree} similar signed graphs. Such a switching was first studied as a tool for generating cospectral graphs \cite{duttajaco} and, subsequently, cospectral signed graphs \cite{PT}.
\\Every real square matrix $A$ of order $qh$ can be viewed as a block matrix
$\{C_{i,j}\}_{i,j=1,\ldots,q}$ as follows:
\begin{equation}\label{eq:blocchi}
A = \left( \begin{array}{ccccc}
C_{1,1} & C_{1,2} & \ldots & C_{1,q-1} & C_{1,q} \\
    {C_{2,1}} & C_{2,2} & \ldots & C_{2, q-1} & C_{2,q} \\
    \vdots & \vdots & \ddots &  \vdots & \vdots \\
    {C_{q-1,1}} & {C_{q-1, 2}} & \ldots & C_{q-1,q-1} & C_{q-1,q} \\
    {C_{q,1}} & {C_{q,2}} & \ldots & {C_{q,q-1}} & C_{q,q}
    \end{array}   \right),
 \end{equation}
where each block $C_{i,j}$ is a square matrix of order $h$.
Given such a  block structure, one can define a new matrix,
called the \emph{partial transpose} of $A$, denoted by
$A^\tau$, as follows:
\begin{equation*}
\begin{split}
A^\tau &= \left( \begin{array}{ccccc}
{C_{1,1}}^t & {C_{1,2}}^t & \ldots & {C_{1,q-1}}^t & {C_{1,q}}^t \\
    {C_{2,1}}^t & {C_{2,2}}^t & \ldots & {C_{2, q-1}}^t & {C_{2,q}}^t \\
    \vdots & \vdots & \ddots &  \vdots & \vdots \\
    {C_{q-1,1}}^t & {C_{q-1, 2}}^t & \ldots & {C_{q-1,q-1}}^t & {C_{q-1,q}}^t \\
    {C_{q,1}}^t & {C_{q,2}}^t & \ldots & {C_{q,q-1}}^t & {C_{q,q}}^t
    \end{array}   \right).
\end{split}
 \end{equation*}
 Given a signed graph $\Gamma$ together with a block structure as in Eq.~\eqref{eq:blocchi} for its adjacency matrix
$A_\Gamma$  (which can be specified by an ordering of the vertices and a divisor of the order of the graph indicating the block size), the partial transpose $\Gamma^\tau$ is defined as the signed graph whose adjacency matrix is $A_\Gamma^\tau$.
 In \cite{duttajaco,PT}, sufficient conditions were given under which the partial transpose yields pairs of cospectral graphs and pairs of cospectral signed graphs, respectively. \\ We now give conditions ensuring that this operation preserves (net-)degree similarity. We begin with an observation, which is in fact a special case of Proposition~\ref{prop:0}.
 \begin{remark}\label{rem:pt}
Let $\Gamma$ be a signed graph whose adjacency matrix admits the block structure of Eq.~\eqref{eq:blocchi}. If the following conditions hold:
\begin{itemize}
\item all vertices belonging to the same block have the same (net-)degree;
\item there exists an invertible real matrix $V$ of order $h$ such that
$V^{-1}C_{i,j}V = {C_{i,j}}^t$ for every $i,j=1,\ldots,q$;
\item the (net-)degrees of $\Gamma^\tau$ coincide with those of $\Gamma$;
\end{itemize}
then $\Gamma$ and $\Gamma^\tau$ are (net-)degree similar. In this case, the matrix $M=\diag(V,\ldots,V)$ plays the role described in Proposition~\ref{prop:0}.
Observe that a necessary condition for $\Gamma$ and $\Gamma^\tau$ not to be switching isomorphic is the non-existence of a signed permutation simultaneously conjugating all blocks to their transposed (see \cite{PT}, where the self-converse property of signed digraphs is studied in connection with this problem).
\end{remark}
We now present some situations in which this approach can be successfully applied.
\begin{proposition}\label{prop:norm}
Let $\Gamma$ be a signed graph whose adjacency matrix is viewed with the block structure of Eq.~\eqref{eq:blocchi}.
Assume that all blocks $\{C_{i,j}\}$ belong to a family of normal and commuting matrices, and that the vertices within the same block have the same degree.
Then $\Gamma$ and $\Gamma^\tau$ are degree similar.
\end{proposition}
\begin{proof}
Under the assumptions, there exists a unitary complex matrix $U$ such that
$U^{-1}C_{i,j}U={C_{i,j}}^t$
for all $i,j=1,\ldots,q$
(see, for example, the proof of \cite[Theorem~3.2]{PT}).
Then, by \cite[Theorem~5]{V}, there exists a real orthogonal matrix $V$ such that
$V^{t}C_{i,j}V={C_{i,j}}^t$
for all $i,j=1,\ldots,q$. Let us put  $Q=\diag(V,\ldots,V)$. One can easily check that $A_{\Gamma^\tau}=Q^{t}A_{\Gamma}Q$. It remains to show that the degree of each vertex is the same in
$\Gamma$ and $\Gamma^\tau$. Observe that if $B$ is a real normal matrix of order $h$, then, for every
$r=1,\ldots,h,$
$$
\sum_{s=1}^h (B)_{sr}^2
=
\sum_{s=1}^h (B^t)_{rs}(B)_{sr}
=
(B^t B)_{rr}
=
(BB^t)_{rr}
=
\sum_{s=1}^h (B)_{rs}(B^t)_{sr}
=
\sum_{s=1}^h (B)_{rs}^2.
$$
Since the block matrices $\{C_{i,j}\}$ have entries in $\{-1,0,1\}$, it follows that, for every block and for every $r=1,\ldots,h$, the $r$th column contains the same number of nonzero entries as the $r$th row. For every vertex, its degree in $\Gamma$ is the sum of the numbers of nonzero entries in the corresponding rows of the blocks, whereas its degree in $\Gamma^\tau$ is the sum of the numbers of nonzero entries in the corresponding columns of the same blocks. Therefore, the degree of each vertex is the same in $\Gamma$ and in $\Gamma^\tau$. The conclusion then follows from
Proposition~\ref{prop:0}.
\end{proof}
A special case in which the assumptions of Proposition~\ref{prop:norm}
are satisfied arises when the blocks are taken from a family of
commuting normal matrices with a constant number of nonzero entries in
each row (and hence in each column). This is essentially the setting of
\cite[Theorem~3.3]{PT}, where in addition each block is required to have
equal row and column sums, an assumption needed there only to deal with
the all-ones matrix $J$; the argument in that proof implicitly yields a
simultaneous conjugation of the adjacency and degree matrices, which
Proposition~\ref{prop:norm} makes explicit in terms of degree similarity.

 In the strategy outlined in Remark~\ref{rem:pt}, it is not actually necessary for the blocks to be normal. In what follows we present  some of the possible cases.
\begin{proposition}\label{prop:singdeg}
Let $\Gamma$ be a signed graph whose adjacency matrix is viewed with the block structure of Eq.~\eqref{eq:blocchi}. Assume that every block belongs to the set
$\{\pm I, O, \pm B, \pm B^t\}$,
where $B$ is a square matrix of order $h$ for which there exists a unitary matrix $U$ satisfying
$U^\ast B U=B^t$.
\\ If every row and every column of $B$ have the same number of nonzero entries, then $\Gamma$ and $\Gamma^\tau$ are degree similar.
 If every row and every column  have the same sum, then $\Gamma$ and $\Gamma^\tau$ are net-degree similar.
\end{proposition}
\begin{proof}
By \cite[Theorem~5]{V}, there exists a real orthogonal matrix $V$ such that
$V^tBV=B^t$ and therefore $V^tB^tV=B$. This easily implies that $V^{t}C_{i,j}V={C_{i,j}}^t$
for all $i,j=1,\ldots,q$ and $A_{\Gamma^\tau}=Q^{t}A_{\Gamma}Q$, where  $Q=\diag(V,\ldots,V)$.
If the number of nonzero entries in each row and column of $B$ and $B^t$ is constant, then the degree of any vertex within the same part is the same in $\Gamma$ and is preserved under the partial transpose $\Gamma^\tau$. The same conclusion holds for the net-degree in the case where the sum of the entries in each row or column of $B$ and $B^t$ is constant.
The result then follows from Proposition~\ref{prop:0}.
\end{proof}

\begin{example}
The signed graph and its partial transpose described in \cite[Example~3.5]{PT} have adjacency matrices whose blocks are normal, commuting, and have constant row sums; hence, by Proposition~\ref{prop:singdeg}, they are net-degree similar. Moreover, it is shown in \cite[Example~5.3]{PT} that these graphs are not switching isomorphic.
\\ The unsigned graph $P_3(T_{11})$ and its non-isomorphic partial transpose $P_3(T_{11})^\tau$ from \cite[Proposition~5.4]{PT} are degree similar (and hence net-degree similar) by the same argument.
\end{example}

 \begin{proposition}\label{prop:singnetmeno}
Let $\Gamma$ be a signed graph whose adjacency matrix is viewed with the block structure of Eq.~\eqref{eq:blocchi}. Assume that every block belongs to the set
$\{\pm I, O, \pm B, \pm B^t\}$,
where $B$ is a square matrix of order $h$ for which there exists a unitary matrix $U$ satisfying
$U^\ast B U=B^t$.
If, in every row of the block structure in Eq.~\eqref{eq:blocchi}, the number of occurrences of $B$ (resp.\ $B^t$) equals that of $-B$ (resp.\ $-B^t$), then $\Gamma$ and $\Gamma^\tau$ are net-degree similar.
\end{proposition}
\begin{proof}
As in the proof of Proposition~\ref{prop:singdeg},
$A_{\Gamma^\tau}=Q^tA_\Gamma Q$ with $Q=\diag(V,\ldots,V)$. In every block
row the contributions of $\pm B$ and of $\pm B^t$ to the net-degrees
cancel out, so the net-degree of each vertex in the $i$-th block of
vertices equals the number of blocks $I$ minus the number of blocks $-I$
in the $i$-th block row, both in $\Gamma$ and in $\Gamma^\tau$. The claim
follows from Proposition~\ref{prop:0}.
\end{proof}

\begin{example}\label{exa:5}
It can be checked that the matrix
$$
B =
\begin{pmatrix}
0 & 0 & 0 \\
0 & 0 & 0 \\
1 & -1 & 1
\end{pmatrix}
$$
is not normal and the sum of the entries along each of its rows is not constant; however, it is conjugated to its transpose by a unitary matrix (see \cite[Example 4.14]{PT}). The adjacency matrices of the signed graph $\Gamma$ depicted on the left
of Fig.~\ref{fig:5} and of its partial transpose $\Gamma^\tau$, depicted
on the right, are
$$
A_\Gamma=\begin{pmatrix}
O      & B      & -B     & O      & O \\
B^t    & O      & -B^t   & O      & I \\
-B^t   & -B     & O      & B      & B^t \\
O      & O      & B^t    & O      & -B^t \\
O      & I      & B      & -B     & O
\end{pmatrix},
\qquad
A_{\Gamma^\tau}=\begin{pmatrix}
O      & B^t    & -B^t   & O      & O \\
B      & O      & -B     & O      & I \\
-B     & -B^t   & O      & B^t    & B \\
O      & O      & B      & O      & -B \\
O      & I      & B^t    & -B^t   & O
\end{pmatrix}.
$$
The hypotheses of Proposition~\ref{prop:singnetmeno} are satisfied, so
the signed graphs $\Gamma$ and $\Gamma^\tau$ are net-degree similar.
On the other hand, it is possible to see that $\Gamma$ and $\Gamma^\tau$ are
not switching isomorphic: a possible way is to observe that any isomorphism
must map the triangles on vertices $9,5,3$ and $9,4,3$ of $\Gamma$, which are
both balanced (and are the only two triangles whose vertices are adjacent to
a pendant vertex), to the triangles on vertices $9,12,13$ and $9,12,14$ of
$\Gamma^\tau$, which are both unbalanced.
 \end{example}

 \begin{figure}
\begin{centering}
\begin{tikzpicture}[scale=0.75,
  vertex/.style={circle, fill=black, minimum size=5pt, inner sep=0pt},
  pos edge/.style={thin},
  neg edge/.style={thin, dashed},
  lbl/.style={draw=none, fill=none, font=\small}
]

\node[vertex,label={[font=\scriptsize]left:$1$}] (1) at (1.5, -1) {};
\node[vertex,label={[font=\scriptsize]right:$2$}] (2) at (6.5, -1) {};
\node[vertex,label={[font=\scriptsize]left:$3$}] (3) at (1, 0) {};

\node[vertex,label={[font=\scriptsize]below right:$9$}] (9) at (4, 0) {};
\node[vertex,label={[font=\scriptsize]right:$15$}] (15) at (7, 0) {};

\node[vertex,label={[font=\scriptsize]below:$7$}] (7) at (3.3, -1.7) {};
\node[vertex,label={[font=\scriptsize]below right:$8$}] (8) at (4.7, -1.7) {};

\node[vertex,label={[font=\scriptsize]left:$5$}] (5) at (3, 2.7) {};
\node[vertex,label={[font=\scriptsize]above:$11$}] (11) at (5.5, 3) {};

\node[vertex,label={[font=\scriptsize] left:$4$}] (4) at (2, 2.7) {};
\node[vertex,label={[font=\scriptsize]below:$6$}] (6) at (4, -0.7) {};

\node[vertex,label={[font=\scriptsize]above:$10$}] (10) at (4.3, 2.7) {};
\node[vertex,label={[font=\scriptsize]above:$12$}] (12) at (6.7, 2.7) {};

\node[vertex,label={[font=\scriptsize]above:$13$}] (13) at (2, 3.2) {};
\node[vertex,label={[font=\scriptsize]above:$14$}] (14) at (3, 3.2) {};

\node[lbl, right] at (7.8,0) {$$};

\draw[pos edge] (3) -- (4);
\draw[neg edge] (3) -- (5);
\draw[neg edge] (3) -- (9);
\draw[neg edge] (3) -- (7);
\draw[pos edge] (13) -- (4);
\draw[neg edge] (9) -- (4);
\draw[pos edge] (9) -- (5);
\draw[pos edge] (14) -- (5);
\draw[pos edge] (3) -- (6);
\draw[pos edge] (15) -- (6);
\draw[neg edge] (6) -- (9);
\draw[pos edge] (9) -- (10);
\draw[pos edge] (9) -- (12);
\draw[pos edge] (11) -- (15);
\draw[pos edge] (9) -- (15);
\draw[neg edge] (10) -- (15);
\draw[neg edge] (12) -- (15);
\draw[neg edge] (11) -- (9);
\draw[neg edge] (8) -- (15);
\draw[pos edge] (8) -- (3);
\draw[pos edge] (7) -- (15);

\end{tikzpicture}
\begin{tikzpicture}[scale=0.75,
  vertex/.style={circle, fill=black, minimum size=5pt, inner sep=0pt},
  pos edge/.style={thin},
  neg edge/.style={thin, dashed},
  lbl/.style={draw=none, fill=none, font=\small}
]
\node[vertex,label={[font=\scriptsize]left:$10$}] (1) at (1.5, -1) {};
\node[vertex,label={[font=\scriptsize]right:$11$}] (2) at (6.5, -1) {};
\node[vertex,label={[font=\scriptsize]left:$12$}] (3) at (1, 0) {};

\node[vertex,label={[font=\scriptsize]below right:$9$}] (9) at (4, 0) {};
\node[vertex,label={[font=\scriptsize]right:$6$}] (15) at (7, 0) {};

\node[vertex,label={[font=\scriptsize]below:$7$}] (7) at (3.3, -1.7) {};
\node[vertex,label={[font=\scriptsize]below right:$8$}] (8) at (4.7, -1.7) {};

\node[vertex,label={[font=\scriptsize]left:$14$}] (5) at (3, 2.7) {};
\node[vertex,label={[font=\scriptsize]above:$2$}] (11) at (5.5, 3) {};

\node[vertex,label={[font=\scriptsize] left:$13$}] (4) at (2, 2.7) {};
\node[vertex,label={[font=\scriptsize]below:$15$}] (6) at (4, -0.7) {};

\node[vertex,label={[font=\scriptsize]above:$1$}] (10) at (4.3, 2.7) {};
\node[vertex,label={[font=\scriptsize]above:$3$}] (12) at (6.7, 2.7) {};

\node[vertex,label={[font=\scriptsize]above:$4$}] (13) at (2, 3.2) {};
\node[vertex,label={[font=\scriptsize]above:$5$}] (14) at (3, 3.2) {};


\draw[neg edge] (3) -- (4);
\draw[pos edge] (3) -- (5);
\draw[pos edge] (3) -- (9);
\draw[pos edge] (3) -- (7);
\draw[pos edge] (13) -- (4);
\draw[pos edge] (9) -- (4);
\draw[neg edge] (9) -- (5);
\draw[pos edge] (14) -- (5);
\draw[neg edge] (3) -- (6);
\draw[pos edge] (15) -- (6);
\draw[pos edge] (6) -- (9);
\draw[neg edge] (9) -- (10);
\draw[neg edge] (9) -- (12);
\draw[neg edge] (11) -- (15);
\draw[neg edge] (9) -- (15);
\draw[pos edge] (10) -- (15);
\draw[pos edge] (12) -- (15);
\draw[pos edge] (11) -- (9);
\draw[pos edge] (8) -- (15);
\draw[neg edge] (8) -- (3);
\draw[neg edge] (7) -- (15);

\end{tikzpicture}

\end{centering}
\caption{ The  net-degree similar  graphs $\Gamma$ and $\Gamma^\tau$ of Example \ref{exa:5}.}\label{fig:5}
\end{figure}
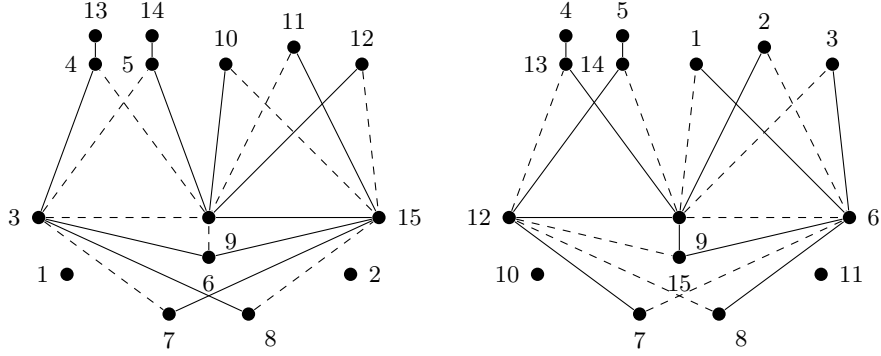

 \section{Block extensions}\label{sec:6}
 In this final section we present a strategy for generating new pairs of
(net-)degree similar graphs from existing ones, exploiting their block
structure.

Let $\Gamma_1$ and $\Gamma_2$ be two (net-)degree similar signed graphs
of order $h$, with vertices ordered in nondecreasing order of
(net-)degree, so that $D^{(\pm)}_{\Gamma_1}=D^{(\pm)}_{\Gamma_2}$ and
$M=\diag(M_1,\ldots,M_k)$ is an invertible block diagonal matrix with
$M^{-1}A_{\Gamma_1}M=A_{\Gamma_2}$, as in Proposition~\ref{prop:-1}.
Let $\overline{\Gamma}_1$ be a signed graph of order $qh$ whose adjacency
matrix is viewed with the block structure of Eq.~\eqref{eq:blocchi},
and put $\overline{M}:=\diag(M,\ldots,M)=I_q\otimes M$. The idea is the
following: if every block $C_{i,j}$ is a polynomial
$p_{i,j}(A_{\Gamma_1},D^{(\pm)}_{\Gamma_1})$, then
$\overline{M}^{-1}A_{\overline{\Gamma}_1}\overline{M}$ is the block matrix with
blocks $p_{i,j}(A_{\Gamma_2},D^{(\pm)}_{\Gamma_2})$; whenever this is the
adjacency matrix of a signed graph $\overline{\Gamma}_2$ and the
(net-)degree is constant on each block of vertices and unchanged from
$\overline{\Gamma}_1$ to $\overline{\Gamma}_2$, Proposition~\ref{prop:0}
yields that $\overline{\Gamma}_1$ and $\overline{\Gamma}_2$ are
(net-)degree similar. We describe in Proposition~\ref{prop:b1} a first situation in which this happens.

Given a signed graph $\Gamma$ of order $h$, consider an ordering of its vertices in
nondecreasing order of (net-)degree, with distinct (net-)degrees
$d^{(\pm)}_1<d^{(\pm)}_2<\cdots<d^{(\pm)}_k$. Therefore
$D^{(\pm)}_{\Gamma}=\diag(d^{(\pm)}_1 I_{n_1},\ldots,d^{(\pm)}_k I_{n_k})$,
where $n_i$ is the number of vertices whose (net-)degree is $d_i^{(\pm)}$, for $i=1,\ldots,k$.
We denote by $\bold e_i$ the vector of length $h$ whose entries are equal to
$1$ at the coordinates corresponding to vertices of (net-)degree $d^{(\pm)}_i$ and
equal to $0$ elsewhere.
We consider the $k$ matrices
\begin{equation}\label{eq:E}
E^{(\pm)}_\Gamma(i)=\diag(\bold e_i)=\diag(O,\ldots,O,I_{n_i},O,\ldots,O),\qquad i=1,\ldots,k,
\end{equation}
such that $D^{(\pm)}_{\Gamma}=\sum_{i=1}^k  d_i^{(\pm)} E^{(\pm)}_\Gamma(i)$.  We denote by $\mathcal{E}^{(\pm)}_{\Gamma}$ the set of all (diagonal)
matrices obtained by taking linear combinations of
$E^{(\pm)}_{\Gamma}(1),\ldots,E^{(\pm)}_{\Gamma}(k)$ with coefficients in
$\{0,\pm1\}$. Observe that every matrix $E\in\mathcal{E}^{(\pm)}_{\Gamma}$ can be expressed
as a polynomial in $D^{(\pm)}_{\Gamma}$. To see that, it is enough to choose a one-variable
polynomial $p$ that maps each $d_i^{(\pm)}$ to the coefficient in
$\{0,\pm 1\}$ of $E^{(\pm)}_{\Gamma}(i)$ in the linear combination defining
$E$.
\begin{proposition}\label{prop:b1}
Let $\Gamma_1$ and $\Gamma_2$ be two (net-)degree similar signed graphs, whose
vertices are ordered in nondecreasing order of (net-)degree. Let $\overline{\Gamma}_1$ be a signed graph whose adjacency matrix $A_{\overline{\Gamma}_1}$ has the
block structure as in Eq.~\eqref{eq:blocchi}, with each block belonging
to
$$\{\pm A_{\Gamma_1},O\}\cup\mathcal{E}^{(\pm)}_{\Gamma_1}.$$
Let $\overline{\Gamma}_2$ be the signed graph obtained by replacing  every block of the form $\pm A_{\Gamma_1}$ with
$\pm A_{\Gamma_2}$ in
$A_{\overline{\Gamma}_1}$. Then $\overline{\Gamma}_1$ and $\overline{\Gamma}_2$ are
(net-)degree similar.
\end{proposition}
\begin{proof}
Since $\Gamma_1$ and $\Gamma_2$ are (net-)degree similar, Proposition~\ref{prop:-1}
provides an invertible matrix $M$ with the block structure described therein
such that
$$
M^{-1}A_{\Gamma_1}M=A_{\Gamma_2},\qquad
M^{-1}D^{(\pm)}_{\Gamma_1}M=D^{(\pm)}_{\Gamma_2}
=D^{(\pm)}_{\Gamma_1}.
$$
Using the block structure of $M$, it is immediate that every matrix
$E\in\mathcal{E}^{(\pm)}_{\Gamma_1}=\mathcal{E}^{(\pm)}_{\Gamma_2}$ is fixed by
the conjugation with $M$, namely $M^{-1}EM=E$. Hence, the matrix
$\overline{M}$ conjugates $A_{\overline{\Gamma}_1}$ into
$A_{\overline{\Gamma}_2}$.
The proof is completed by applying Proposition~\ref{prop:0} and observing
that $\overline{\Gamma}_1$ and $\overline{\Gamma}_2$ have the same
(net-)degrees, which are constant on the blocks of
$\overline{M}$.
\end{proof}
Recall that a rooted signed graph is a signed graph where a special vertex, called the root, has been fixed.
\begin{example}\label{exa:rooted}
Let $\Gamma$ be a signed graph and let $\mathcal K=(K_v)_{v\in V_\Gamma}$
be a family of rooted signed graphs. The \emph{rooted product}
$\Gamma(\mathcal K)$, introduced in \cite{GMr} for unsigned graphs, is
the signed graph obtained by identifying the root of $K_v$ with the
vertex $v$, for every $v\in V_\Gamma$. Let $\Gamma_1$ and $\Gamma_2$ be
(net-)degree similar, with vertices ordered in nondecreasing order of
(net-)degree, and let $K^{(1)},\ldots,K^{(k)}$ be rooted signed graphs
with the same vertex set $\{1,\ldots,q\}$ and root $1$; let
$\Gamma_j(\mathcal K)$ be the rooted product in which a copy of
$K^{(i)}$ is attached at every vertex of the $i$-th (net-)degree class
of $\Gamma_j$, with $j=1,2$ and $i=1,\ldots, k$. With a suitable ordering of the  vertices of $\Gamma_j(\mathcal K)$
(by taking first $V_{\Gamma_j}$) one has
$$
A_{\Gamma_j(\mathcal K)}=E_{11}\otimes A_{\Gamma_j}
+\sum_{i=1}^{k}A_{K^{(i)}}\otimes E^{(\pm)}_{\Gamma_1}(i),
$$
where $E_{11}$ denotes the $q\times q$ matrix whose only nonzero entry
is a $1$ in position $(1,1)$. Hence $A_{\Gamma_j(\mathcal K)}$ has the
block structure of Eq.~\eqref{eq:blocchi}, with blocks of type $A_{\Gamma_j}$ and sums of some
$E^{(\pm)}_{\Gamma_1}(i)$'s, which belong to $\mathcal E^{(\pm)}_{\Gamma_1}$,
so that $\Gamma_1(\mathcal K)$ and $\Gamma_2(\mathcal K)$ are (net-)degree
similar by Proposition~\ref{prop:b1}. The assumption that the $K^{(i)}$
have the same order is not restrictive for degree similarity, since one may
add isolated vertices to them, which only adds isolated vertices to
$\Gamma_j(\mathcal K)$ (Remark~\ref{rem:delete}). In the unsigned case,
this construction contains the rooted product at vertices whose degree is different
from that of all other vertices \cite[Lemma~18]{GS} and the attachment
of $s_i$ pendant vertices to every vertex of degree $d_i$
\cite[Lemma~20]{GS}, which corresponds to choose as $K^{(i)}$  a star graph on $s_i+1$ vertices rooted at
the centre.
\\ As a special case, take the degree similar signed graphs $C_{2n}^+$
and $C_n^+\cup C_n^-$ of Example~\ref{ex:uni} and attach a copy of the
same rooted tree $T$ at every vertex. The resulting signed graphs are
degree similar, but the first is connected and unicyclic while the
second is not. This yields infinitely many connected unicyclic signed
graphs which are not determined by their degree similarity class.
Similarly, attaching a copy of the same rooted signed tree $T$ at every
vertex of the net-degree similar unicyclic signed graphs $\Gamma$ and
$\Gamma^\pi$ of Example~\ref{exa:uni}
(or, more generally, three rooted
signed trees of the same order, one for each of the three net-degree
classes)
yields infinitely many pairs of net-degree similar unicyclic signed
 graphs with non-isomorphic underlying graphs.
\end{example}

Another case of interest arises, in the specific context of net-degree similarity,
when the signed graph  is such that the null space of its
net-Laplacian  has dimension $1$. Some families of signed graphs with this
property are studied in \cite{null}; for instance, every signed tree satisfies this
condition, as does every connected unsigned graph (or, equivalently, every
connected all-positive or all-negative signed graph).
\\We introduce some additional possible kinds of blocks for this setting. In analogy with Eq.~\eqref{eq:E}, we define the following $k$ matrices:
\begin{equation*}\label{eq:J}
J^{\pm}_{\Gamma}(i)=\mathbf{e}_i\mathbf{e}_i^t
=\diag(O,\ldots,O,J_{n_i},O,\ldots,O),\qquad i=1,\ldots,k,
\end{equation*}
where $J_{n_i}$ denotes the $n_i\times n_i$ all-ones matrix.
We denote by
$\mathcal{J}^{\pm}_{\Gamma}$ the set of all linear combinations of $J^{\pm}_{\Gamma}(1),\ldots,J^{\pm}_{\Gamma}(k)$  with
coefficients in $\{0,\pm1\}$.

\begin{theorem}\label{prop:b2}
Let $\Gamma_1$ and $\Gamma_2$ be two net-degree similar signed graphs whose
vertices are ordered in nondecreasing order of net-degree, and assume that
$\dim(\ker L^\pm_{\Gamma_1})=1$. Let $\overline{\Gamma}_1$ be a signed graph
whose adjacency matrix  $A_{\overline{\Gamma}_1}$ has the block structure described in
Eq.~\eqref{eq:blocchi}, with each block belonging to
$$\{\pm A_{\Gamma_1},\pm J,O\}\cup
\mathcal{E}^{\pm}_{\Gamma_1}\cup\mathcal{J}^{\pm}_{\Gamma_1}.$$
Let $\overline{\Gamma}_2$ be the signed graph obtained from
$\overline{\Gamma}_1$ by replacing every block of the form
$\pm A_{\Gamma_1}$  with
$\pm A_{\Gamma_2}$ in $A_{\overline{\Gamma}_1}$. Then $\overline{\Gamma}_1$ and
$\overline{\Gamma}_2$ are net-degree similar.
\end{theorem}
\begin{proof}
In view of Proposition~\ref{prop:b1}, it is enough to show that the block
matrix $M$ provided by Proposition~\ref{prop:-1} commutes with the additional
blocks $\pm J$ and with every matrix in $\mathcal{J}^{\pm}_{\Gamma_1}$.
The matrix $J$ commutes with $L^\pm_{\Gamma_1}$, and therefore it can be
expressed as a polynomial in $L^\pm_{\Gamma_1}$ (it is enough to choose a
polynomial that maps $0$ to $n$ and every other eigenvalue of the
net-Laplacian to $0$).
Hence,  $J$  is also a two-variable polynomial in
$A_{\Gamma_1}$ and $D^\pm_{\Gamma_1}$.
Moreover, the same polynomial evaluated in
$A_{\Gamma_2}$ and $D^\pm_{\Gamma_2}$
again yields the matrix $J$
(see the analogous argument in the proof of
\cite[Lemma~8]{GS}). In particular, this implies that $M$ satisfies $MJ=JM$, that is, it has constant row and
column sums. Moreover, since
$M=\diag(M_1,\ldots,M_k)$, each block $M_i$ also has constant row and column
sums, or equivalently commutes with the corresponding matrix $J_{n_i}$.
Therefore, $M$ commutes with every matrix in
$\mathcal{J}^{\pm}_{\Gamma_1}$, and the proof is complete.
\end{proof}

\begin{example}
Let $\Gamma_1$ and $\Gamma_2$ be two net-degree similar signed graphs, whose
vertices are ordered in nondecreasing order of net-degree, with
$\dim(\ker L^\pm_{\Gamma_1})=1$. Consider the
signed graphs $\overline{\Gamma}_1$ and $\overline{\Gamma}_2$ with
adjacency matrices
$$
A_{\overline{\Gamma}_1}=
\begin{pmatrix}
A_{\Gamma_1} & W\\
W & O
\end{pmatrix},\qquad
A_{\overline{\Gamma}_2}=
\begin{pmatrix}
A_{\Gamma_2} & W\\
W & O
\end{pmatrix},
$$
where $W\in\mathcal{J}^{\pm}_{\Gamma_1}$. By Theorem~\ref{prop:b2}
the signed graphs $\overline{\Gamma}_1$ and $\overline{\Gamma}_2$ are
net-degree similar. In other words, for each net-degree class chosen in
the linear combination defining $W$, one adds as many new vertices as
the vertices of that class, each of them adjacent, all positively or all
negatively, to every vertex of the class; the remaining new vertices are
isolated. In the unsigned case, where $\Gamma_1$ and $\Gamma_2$ are
connected and $W$ has coefficients in $\{0,1\}$, deleting the isolated
vertices (Remark~\ref{rem:delete}) gives the construction of
\cite[Lemma~21]{GS}.
\end{example}

We now prove a result concerning net-regular signed trees, which provides
a further kind of admissible blocks. Given a signed tree $T$ of order
$h$, its underlying graph is bipartite. We denote by $\mathbf{v}_T$ a
vector whose coordinates are indexed by the vertices of $T$, taking value
$1$ on one part of the bipartition and $-1$ on the other, and we put
$K_T=\mathbf{v}_T\mathbf{v}_T^t$.
\begin{corollary}\label{coro:b1}
Let $T_1$ and $T_2$ be two cospectral $-1$-net-regular signed trees. Let $\overline{T_1}$ be a signed graph
whose adjacency matrix $A_{\overline{T_1}}$ has the block structure described in
Eq.~\eqref{eq:blocchi}, with each block belonging to
$$\{\pm A_{T_1},\pm K_{T_1},\pm J,O\}\cup
\mathcal{E}^{\pm}_{T_1}\cup\mathcal{J}^{\pm}_{T_1}.$$
Let $\overline{T}_2$ be the signed graph obtained from
$\overline{T}_1$ by replacing, in $A_{\overline{T}_1}$, every block of the
form $\pm A_{T_1}$ with $\pm A_{T_2}$ and every block of the form
$\pm K_{T_1}$ with $\pm K_{T_2}$. Then $\overline{T}_1$ and
$\overline{T}_2$ are net-degree similar.
\end{corollary}
\begin{proof}
Since $T_i$ is $-1$-net-regular, it is immediate to verify that
$\mathbf v_{T_i}$ is an eigenvector of $A_{T_i}$ with eigenvalue $1$,
whereas $\mathbf 1$ is an eigenvector with eigenvalue $-1$. Hence
$\mathbf v_{T_i}^t\mathbf 1=0$, so that $K_{T_i}$ has zero row sums and
the blocks $\pm K_{T_i}$ do not affect the net-degrees of
$\overline T_i$. Moreover, as already observed,
$\dim(\ker L^\pm_{T_1})=1$.
\\Hence, we are in the setting of Theorem~\ref{prop:b2}, and it remains
only to verify how $K_{T_1}$ behaves under conjugation by the matrix $M$
provided by Proposition~\ref{prop:-1}.
\\The matrix $K_{T_1}$ commutes with $A_{T_1}$, is symmetric, and its
only nonzero eigenvalue is $h$, with corresponding eigenvector
$\mathbf v_{T_1}$. Moreover, the eigenvalue $1$ of $A_{T_i}$ has
multiplicity $1$, since in a tree the multiplicity of any eigenvalue
admitting an eigenvector with no zero entries is necessarily $1$ (see
\cite[Proposition~1]{F}). Hence there exists a polynomial $p$, for
instance any polynomial satisfying $p(1)=h$ and $p(\lambda)=0$ for every
other eigenvalue $\lambda$ of $A_{T_1}$, such that $p(A_{T_1})=K_{T_1}$.
Since $A_{T_1}$ and $A_{T_2}$ are cospectral, then $p(A_{T_2})=K_{T_2}$,
and therefore $M^{-1}K_{T_1}M=K_{T_2}$.
\end{proof}
A final special case concerns unsigned graphs.
Indeed, in this setting, the null space of the Laplacian is one-dimensional
if and only if the graph is connected. If $\Gamma$ is an unsigned graph, we denote by $\Gamma^c$ its complement
graph.
\begin{corollary}\label{coro:ultimo}
Let $\Gamma_1$ and $\Gamma_2$ be two connected unsigned degree similar  graphs whose
vertices are ordered in nondecreasing order of degree. Let $\overline{\Gamma}_1$ be a signed graph
whose adjacency matrix $A_{\overline{\Gamma}_1}$ has the block structure described in
Eq.~\eqref{eq:blocchi}, with each block belonging to
$$\{\pm A_{\Gamma_1},\pm A_{\Gamma^c_1}, \pm J,O\}\cup
\mathcal{E}^\pm_{\Gamma_1}\cup\mathcal{J}^\pm_{\Gamma_1}.$$
Let $\overline{\Gamma}_2$ be the signed graph obtained from
$\overline{\Gamma}_1$ by replacing, in $A_{\overline{\Gamma}_1}$, every block of the
form $\pm A_{\Gamma_1}$ with $\pm A_{\Gamma_2}$ and every block of the form
$\pm A_{\Gamma^c_1}$ with $\pm A_{\Gamma^c_2}$. Then $\overline{\Gamma}_1$ and
$\overline{\Gamma}_2$ are net-degree similar.
\end{corollary}
\begin{proof}
Since $\Gamma_1$ is connected, the assumptions of
Theorem~\ref{prop:b2} are satisfied. The result then follows
from the identity $A_{\Gamma_i^c}=J-I-A_{\Gamma_i}$, for $i=1,2$.
\end{proof}

\begin{example}
Given an unsigned graph $\Gamma$ of order $h$, the \emph{complementary prism} of
$\Gamma$ is the graph obtained by taking the disjoint union of $\Gamma$ and
its complement $\Gamma^c$, and adding $h$ edges joining each vertex of
$\Gamma$ to its corresponding vertex in $\Gamma^c$ (see \cite{prism0} for the definition of the complementary prism
and \cite{prism1} for its spectral properties).
If $\Gamma_1$ and $\Gamma_2$ are two connected unsigned graphs that are
degree similar, then their respective
complementary prisms are also degree similar.
Indeed, the adjacency matrices of
the complementary prisms are
$
\begin{pmatrix}
A_{\Gamma_1} & I\\
I & A_{\Gamma_1^c}
\end{pmatrix}$
and $
\begin{pmatrix}
A_{\Gamma_2} & I\\
I & A_{\Gamma_2^c}
\end{pmatrix},
$
so that
Corollary~\ref{coro:ultimo} applies.
\end{example}

\section*{Acknowledgements and declarations}
The authors acknowledge the support of the INdAM-GNSAGA project \lq\lq Between groups and graphs\rq\rq, CUP E53C25002010001. This research was partially supported by the Engineering Department of Cusano University of Rome.

\end{document}